\documentclass[11pt,reqno,twoside]{amsart}
\usepackage{graphicx}
\usepackage{amssymb}
\usepackage{epstopdf}
\usepackage[asymmetric,top=3.5cm,bottom=4.3cm,left=3.1cm,right=3.1cm]{geometry}
\usepackage{bm}
\usepackage{algorithm}
\usepackage{algpseudocode}
\newenvironment{varalgorithm}[1]
{\algorithm\renewcommand{\thealgorithm}{#1}}
{\endalgorithm}

\usepackage{hyperref}
\usepackage{xcolor}
\hypersetup{
  colorlinks,
  linkcolor={blue!80!black},
  urlcolor={blue!80!black},
  citecolor={blue!80!black}
}
\usepackage[capitalize,noabbrev]{cleveref}

\usepackage{booktabs} 
\usepackage{array} 
\usepackage{paralist} 
\usepackage{verbatim} 
\usepackage{subfig} 
\usepackage{tabularx}
\usepackage{amsmath,amsfonts,amsthm,mathrsfs,amssymb,cite}
\usepackage{ulem}

\newtheorem{thm}{Theorem}[section]

\newtheorem{lem}{Lemma}[section]

\theoremstyle{definition}

\theoremstyle{remark}

\newtheorem{rem}{Remark}[section]
\numberwithin{equation}{section}

\numberwithin{equation}{section}

\newcounter{saveeqn}

\newcommand{\eqnref}[1]{(\ref {#1})}

\newcommand{\Bz}{\mathbf{z}}

\newcommand{\bE}{\mathbf{E}}

\newcommand{\bH}{\mathbf{H}}

\newcommand{\Bx}{\mathbf{x}}
\newcommand{\By}{\mathbf{y}}

\newcommand{\bZ}{\mathbf{z}}

\newcommand{\Gl}{\lambda}

\newcommand{\Gs}{\sigma}

\newcommand{\Ge}{\varepsilon}

\newcommand{\Acal}{\mathcal{A}}
\newcommand{\Kcal}{\mathcal{K}}

\newcommand{\Scal}{\mathcal{S}}

\newcommand{\Mcal}{\mathcal{M}}

\newcommand{\Ocal}{\mathcal{O}}

\newcommand{\Ical}{\mathcal{I}}
\newcommand{\Tcal}{\mathcal{T}}
\newcommand{\Ncal}{\mathcal{N}}

\newcommand{\ds}{\displaystyle}

\newcommand{\RR}{\mathbb{R}}

\newcommand{\p}{\partial}

\newcommand{\beq}{\begin{equation}}
\newcommand{\eeq}{\end{equation}}

\DeclareMathAlphabet{\itbf}{OML}{cmm}{b}{it}

\title[Locating multiple geomagnetic anomalies]{Locating multiple magnetized anomalies by geomagnetic monitoring}

\author{Rongliang Chen}
\address{Shenzhen Institutes of Advanced Technology, Chinese Academy of Sciences, Shenzhen, P.~R.~China}
\email{rl.chen@siat.ac.cn}

\author{Youjun Deng}
\address{School of Mathematics and Statistics, Central South University, Changsha, Hunan, P. R. China.}
\email{youjundeng@csu.edu.cn; dengyijun\_001@163.com}

\author{Yang Gao}
\address{School of Mathematics and Statistics, Central South University, Changsha, Hunan, P. R. China.}
\email{gaoyang\_006@126.com}

\author{Jingzhi Li}
\address{Department of Mathematics \& National Center for Applied Mathematics Shenzhen \& SUSTech International Center for Mathematics, Southern University of Science and Technology, Shenzhen 518055, P.~R.~China. }
\email{li.jz@sustech.edu.cn}

\author{Hongyu Liu}
\address{Department of Mathematics, City University of Hong Kong, Hong Kong SAR, China}
\email{hongyliu@cityu.edu.hk}

\begin{document}
\maketitle

\begin{abstract}

The presence of magnetized anomalies in the shell of the Earth interrupts its geomagnetic field. We consider the inverse problem of identifying the anomalies by monitoring the variation of the geomagnetic field. Motivated by the theoretical unique identifiability result in \cite{DLL:19}, we develop a novel numerical scheme of locating multiple magnetized anomalies. 
In our study, we do not assume that the source that generates the geomagnetic field, and the medium configurations of the Earth's core and the magnetized anomalies are a-priori known. The core of the reconstruction scheme is a novel imaging functional whose quantitative behaviours can be used to identify the anomalies. Both rigorous analysis and extensive numerical experiments are provided to verify the effectiveness and promising features of the proposed reconstruction scheme.


\medskip

\medskip



\end{abstract}

\section{Introduction}

\subsection{Problem formulation and setup}\label{sec:setup}

It is widely accepted in geophysics that the Earth is of a multi-layered structure. In the present study, we simplify the geometric modelling of the Earth by assuming that it is of a core-shell structure; see Figure~\ref{fig:earth} for a schematic illustration. We denote by $\Omega_c$ and $\Omega_s$ the core and shell of the Earth, respectively. Let $\Omega:=\Omega_s\cup\overline{\Omega_c}$ denote the Earth. Both  $\Omega$ and $\Omega_c$ are bounded simply-connected $C^2$ domains in $\mathbb{R}^3$. Next we describe the medium configuration of the whole space. Let $\varepsilon_0$ and $\mu_0$ be positive constants which respectively signify the electric permittivity and the magnetic permeability of the outer space $\mathbb{R}^3\backslash\overline{\Omega}$. For the shell $\Omega_s$, we assume that the electric permittivity is described by a positive $L^\infty$ function $\varepsilon_s(\Bx)$, $\Bx\in\Omega_s$, and the magnetic permeability is the same as that in the outer space. We let the medium in the core $\Omega_c$ be characterised by positive $L^\infty$ functions $\varepsilon_c(\Bx)$, $\mu_c(\Bx)$ and $\sigma_c(\Bx)$, which describe the electric permittivity, the magnetic permeability and the electric conductivity, respectively. The electric conductivity of the outer space and the shell is assumed to be identically zero.

\begin{figure}
    \centering
    \includegraphics[width=0.4\textwidth]{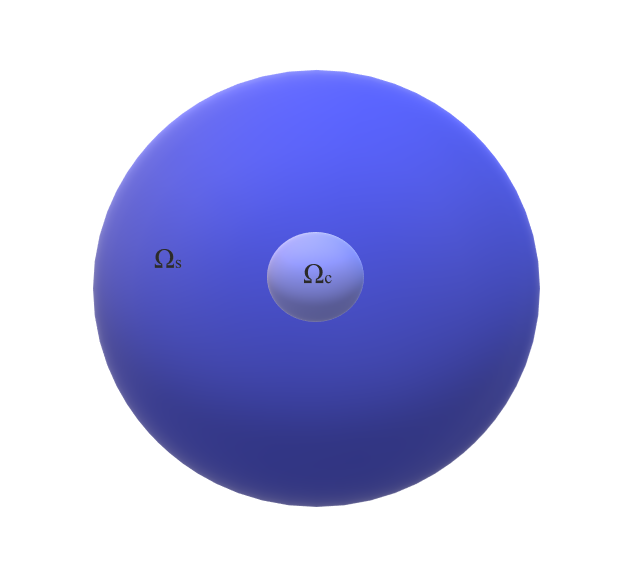}
    \caption{Schematic illustration of the core-shell structure of the Earth.}
    \label{fig:earth}
\end{figure}

It is assumed that the presence of certain magnetized anomalies in the shell $\Omega_s$ interrupts the geomagnetic field of the Earth. Next, we describe the mathematical modelling of such a practical scenario. Suppose that there are $l_0$ magnetized anomalies, which are assumed to be simply-connected Lipschitz domains $B_l, l=1,2,\ldots, l_0$. For each magnetized anomaly $B_l$, its electromagnetic parameters are described by positive constants $\varepsilon_l, \mu_l, \sigma_l$, referring to the electric permittivity, the magnetic permeability and the electric conductivity, respectively. It is important to stress that the magnetic permeability of the magnetized anomaly is different from that of the surrounding medium in the shell. From a practical point of view, the magnetized anomalies are small compared to the size of the Earth. Hence, we suppose that
\beq\label{eq:permeab02}
B_l=\delta U+\Bz_l, \quad l=1, 2, \ldots, l_0,
\eeq
where $\delta\in\mathbb{R}_+$ and $\delta\ll 1$, $U\Subset \Omega_s$ is a bounded Lipschitz domain, and $\Bz_l\in\Omega_s$, $l=1, 2, \ldots, l_0$, are far away from $\p \Omega_s$ such that $\Bx-\Bz_l\gg\delta$, for any $\Bx\in\p \Omega_s$. Note that we have supposed that all the magnetic anomalies are of the same shape and size for the convenience of theoretical analysis. We shall see in the numerical examples in Section~\ref{sec:experiments} that such a restriction is not necessary.

Next, we present the PDE system that governs the propagation of the geomagnetic waves. We use $\bE $ and $\bH $, to represent the electric and magnetic fields in the frequency domain, respectively. Let the functions $\mathbf{v}\in ( L^\infty(\Omega_c))^3$  and $\rho\in L^2(\Omega_c)$ describe the fluid velocity and the charge density of the core, respectively. Moreover, both of them have zero extensions to $\RR^3$. Then the electric and magnetic fields satisfy the following time-harmonic Maxwell system:
\begin{equation}\label{eq:pss01}
\left \{
 \begin{array}{ll}
\nabla\times\bH=-\mathrm{i}\omega\varepsilon\bE+\Gs(\bE+\mu\mathbf{v}\times\bH)  & \mbox{in} \ \ \RR^3,\medskip\\
\nabla\times\bE=\mathrm{i}\omega\mu\bH & \mbox{in} \ \ \RR^3,\medskip\\
\nabla\cdot (\mu\bH)=0,\quad \nabla\cdot (\varepsilon\bE)=\rho  & \mbox{in} \ \ \RR^3,\medskip\\
 \lim_{\|\Bx\|\rightarrow\infty} \|\Bx\| \big( \sqrt{\mu_0}\bH \times\frac{\Bx}{\|\Bx\|}- \sqrt{\varepsilon_0}\bE\big)=\mathbf{0},
 \end{array}
 \right .
 \end{equation}
where $\mathrm{i}:=\sqrt{-1}$ is the imaginary unit, $\omega\in\mathbb{R}_+$ is the angular frequency of the wave propagation, and $\|\cdot\|$ stands for the Euclidean norm of a vector. In \eqref{eq:pss01}, the last limit is known as the Silver-M\"uller radiation condition which characterizes the outward radiating solutions and it holds uniformly in all directions $\hat{\Bx}:=\Bx/\|\Bx\|$ (cf. \cite{CK, DHU:17, Ned}).

Under the setup in our description above, if there is no magnetized anomaly in the shell, the medium parameters $\varepsilon, \mu$ and $\Gs$  in \eqref{eq:pss01} are described by
\begin{equation}\label{eq:paradef01}
\left \{
\begin{split}
\varepsilon(\Bx)&=(\varepsilon_c(\Bx)-\varepsilon_0)\chi(\Omega_c)+(\varepsilon_s(\Bx)-\varepsilon_0)\chi(\Omega_s)+\varepsilon_0,\\
\mu(\Bx)&=(\mu_c(\Bx)-\mu_0)\chi(\Omega_c)+\mu_0,\\
\Gs(\Bx)&=\Gs_c(\Bx)\chi(\Omega_c),\\
\end{split}
\right .
\end{equation}
where, and also in what follows, $\chi$ denotes the characteristic function. On the other hand, if there are $l_0$  magnetized anomalies in the shell, they are described by
\beq\label{eq:paradef02}
\left \{
\begin{split}
\varepsilon(\Bx)=& (\varepsilon_c(\Bx)-\varepsilon_0)\chi(\Omega_c)+(\varepsilon_s(\Bx)-\varepsilon_0)\chi(\Omega_s\setminus\overline{\bigcup_{l=1}^{l_0}B_l}) \\
&+\sum_{l=1}^{l_0} (\varepsilon_l-\varepsilon_0)\chi(B_l)+\varepsilon_0,\\
\quad\mu(\Bx)=& (\mu_c(\Bx)-\mu_0)\chi(\Omega_c)+\sum_{l=1}^{l_0} (\mu_l-\mu_0)\chi(B_l)+\mu_0,\\
\Gs(\Bx)=& \Gs_c(\Bx)\chi(\Omega_c)+\sum_{l=1}^{l_0} \sigma_l\chi(B_l).\\
\end{split}
\right .
\eeq

Hence, the solutions of \eqref{eq:pss01} differ between the two cases. In what follows, we write $(\mathbf{E}_0, \mathbf{H}_0)$ to denote the solution of \eqref{eq:pss01} associated to the parameter configuration in \eqref{eq:paradef01}, and $(\mathbf{E}, \mathbf{H})$ to be the solution of \eqref{eq:pss01} associated to the parameter configuration in \eqref{eq:paradef02}. The geomagnetic inverse problem is to use this difference to determine the magnetized anomalies; that is,
\begin{equation}\label{eq:geom3}
\mathbf{H}(\Bx;\omega)-\mathbf{H}_0(\Bx;\omega)\big|_{(\Bx,\omega)\in\Gamma\times\mathbb{R}_+}\longrightarrow \bigcup_{l=1}^{l_0} B_l,
\end{equation}
where $\mathbf{H}(\cdot; \omega)$ and $\mathbf{H}_0(\cdot; \omega)$ indicate the dependence on the frequency $\omega$, and $\Gamma$ is an open surface located in the outer space $\mathbb{R}^3\backslash\overline{\Omega}$. It is emphasised that for the inverse problem \eqref{eq:geom3}, we shall not assume any a-priori knowledge about the core of the Earth, namely $\varepsilon_c, \mu_c, \sigma_c$ as well as $\rho, \mathbf{v}$ are not assumed to known in advance.

\subsection{Background and discussion of main results}

The geomagnetic field, a.k.a. the Earth's magnetic field, is generated by the motion of electrically conducting fluids in the Earth's outer core; see \cite{BPCo96,Fey,Jacobs,Weiss} and the references cited therein for more background discussion about this aspect. As shall be described in Subsection \ref{sec:setup}, the geomagnetic field varies due to the presence of the magnetized anomalies. By measuring such a geomagnetic variation, one can infer information about the magnetized anomalies, such as the location, the shape, the material parameters and the motion. This magnetic anomaly detecting (MAD) technique is widely used in various fields of practical importance, including the exploration of underground minerals, detection of nuclear submarines, and earthquake prediction. Hence, the development of effective and stable MAD techniques have attracted considerable attention in the literature; see \cite{MAD1,MAD2,MAD3,MAD4,MAD5} and the references cited therein. 

For the geomagnetic inverse problem \eqref{eq:geom3} of the current study, there are several major challenges. First, as discussed earlier, we shall not impose any a-priori knowledge about the core of the Earth, including its medium configuration, namely $\varepsilon_c, \mu_c$ and $\sigma_c$, as well as the interior fluid velocity $\mathbf{v}$ and the charge density $\rho$ of the core. This is highly important from a practical point of view since our knowledge about the Earth's core is limited. This is also a salient feature of our inverse problem \eqref{eq:geom3} compared to many existing ones in the literature, where one usually assumes that most of the information of the underlying PDE system is known except those ones to be recovered. Second, we would like to emphasise that it is directly verified that the inverse problem \eqref{eq:geom3} is nonlinear, and also ill-posed (though heuristically without verification since it is not the focus of our study). Third, we would like to point out the dataset used in \eqref{eq:geom3}. In fact, we shall not make use of the full frequency information, and instead, we shall mainly make use of the low-frequency information. That is, in \eqref{eq:geom3}, $(\hat{\Bx},\omega)\in \hat{\Gamma}\times\mathbb{R}_+$ is actually replaced to be $(\hat{\Bx},\omega)\in \hat{\Gamma}\times(0, \omega_0)$ with $\omega_0\in\mathbb{R}_+$ being a small and fixed constant.  In fact, our study is motivated by the theoretical results in \cite{DLL:19} by two authors of the current article which we shall discuss a bit more details in what follows, and we also refer to \cite{DLL20,DLT} for related theoretical studies on geomagnetic anomaly identifications. 


In \cite{DLL:19}, a rigorous mathematical theory was established for identifying magnetized anomalies by measuring the variation of the geomagnetic field due to the presence of the anomaly. Our numerical study is mainly motivated by that theoretical work. In fact, 
we develop a novel imaging scheme to recover the location of magnetized anomalies by constructing an indicator functional $I(\Bz), \Bz\in\Omega_s$, which can effectively identify the locations of multiple magnetised anomalies. In constructing the imaging functional, we only make use of the static part of the difference of the magnetic fields, namely the leading-part of $\tilde{\mathbf{H}}(\mathbf{x}):=\mathbf{H}(\Bx, \omega)-\mathbf{H}_0(\Bx, \omega)$ with respect to $\omega\rightarrow 0^+$. Based on the low-frequency asymptotics of the geomagnetic fields, the quantitative indicating behaviours of the functional is rigorously justified. 
 One shall see that the indicator function will blow-up when the sampling point is sufficiently close to location of an anomaly and will be bounded when the sampling point is located away from it. Moreover, we only use inner products to construct an indicator function without any complicated procedure. The results of numerical experiments show that our algorithm is efficient and robust. Finally, it is remarked that our method is mainly concerned with the identification of the location of the (small) magnetised anomalies, and we shall not consider the further determination of the material parameters of the anomalies. It is also pointed out that we shall not assume that the material properties of the magnetised anomalies are known in advance. This makes our study practically meaningful for certain applications, say e.g. the submarine detection and the mineral detection.  

The paper is organized as follows. In Section~\ref{sec:pre}, we present preliminary knowledge on the forward problem as well as the layer potential method for our subsequent use. In Section~\ref{sec:main}, we introduce our novel imaging functional as well as verify its desired indicating behaviours. Using this imaging functional, we then propose the imaging scheme for locating the magnetised anomalies. Numerical experiments are presented to verify the effectiveness and efficiency of the proposed method in Section~\ref{sec:experiments}.

\section{Preliminaries}\label{sec:pre}

In this section, we present some preliminary results of the direct problem \eqref{eq:pss01} as well as the layer potential method for our subsequent use.

We first introduce the function spaces
\[
H(\mathrm{curl}, U)=\{f\in (L^2(U))^3|\ \nabla\times f\in (L^2(U))^3\},
\]
and
\[
H_{loc}(\mathrm{curl}, X):=\{f|_U\in H(\mathrm{curl}, U)|\ U\ \ \mbox{is any bounded subdomain of $X$}\}.
\]
There exists a unique pair of solutions $(\bE,\bH)\in H_{loc}(\mathrm{curl}, \mathbb{R}^3)\times H_{loc}(\mathrm{curl}, \mathbb{R}^3)$ to \eqref{eq:pss01}. For  detailed research about the well-posedness of \eqref{eq:pss01}, we refer to \cite{LRX}.

Next, we introduce the layer potential operators which facilitate the integral representations of the solutions to \eqref{eq:pss01}. We refer to \cite{HK07, CK, Ned} for general introductions on layer potential operators and \cite{DHH:17,DHU:17,DLL20} for those related to the current study. Let $\Gamma_k$ be the fundamental solution  to the partial differential operator $\Delta +k^2$, which is given by
\begin{equation}\label{Gk} 
\Gamma_k(\Bx) = -\frac{e^{\mathrm{i}k\|\Bx\|}}{4 \pi \|\Bx\|},~~ \Bx\in\RR^3, \quad \Bx\neq \mathbf{0}.
\end{equation}
Suppose that $D\subset \RR^3$ is a bounded Lipschitz domain. For a given scalar density function $\phi$, the single layer potential operator $\Scal_D^k$ and  the Neumann-Poincar\'e operator $\Kcal_D^{k}$ are respectively given by
\beq\label{eq:layperpt1}
\Scal_{D}^{k}[\phi](\Bx):=\int_{\p D}\Gamma_k(\Bx-\By)\phi(\By) ~\mathrm{d} s_\By,
\eeq
\beq\label{eq:layperpt2}
\Kcal_{D}^{k}[\phi](\Bx):=\mbox{p.v.}\quad\int_{\p D}\frac{\p\Gamma_k(\Bx-\By)}{\p \nu_y}\phi(\By)~\mathrm{d} s_\By,
\eeq
where p.v. stands for the Cauchy principle value and $\nu$ denote the unit outside normal vector to $\p D$. They also have the following mapping property:
\beq
\begin{split}
\Scal_D^k: H^{-1/2}(\p D)&\rightarrow H^{1}(\RR^3\setminus\p D),\\
\Kcal_D^{k}: H^{1/2}(\p D)&\rightarrow H^{1/2}(\p D),\\
\end{split}
\eeq
respectively. Moreover, they satisfy the following trace formula:
\beq \label{eq:trace}
\frac{\p}{\p\nu}\Scal_D^{k}[\phi] \Big|_{\pm} = (\pm \frac{1}{2}I+(\Kcal_{D}^{k})^*)[\phi] \quad \mbox{on} \, \p D,
\eeq
where $(\Kcal_{D}^{k})^*$ is the adjoint operator of $\Kcal_D^{k}$; $\pm$ signify the traces taken from the outside and inside of $D$, and $I$ is the identity operator.  

To derive the integral representation of the Maxwell system, the vectorial layer potentials shall be used. Before that, we introduce the normed spaces of the tangential fields on the boundary $\partial D$:
\begin{align*}
L_T^2(\p D):=\{\Phi\in {L^2(\p D)}^3: \nu\cdot \Phi=0\}.
\end{align*}
Let $\nabla_{\p D}\cdot$ denote the surface divergence. Set
\begin{align*}
\mathrm{TH}({\rm div}, \p D):&=\Bigr\{ {\Phi} \in L_T^2(\partial D):
\nabla_{\partial D}\cdot {\Phi} \in L^2(\partial D) \Bigr\},\\
\mathrm{TH}({\rm curl}, \p D):&=\Bigr\{ {\Phi} \in L_T^2(\partial D):
\nabla_{\partial D}\cdot ({\Phi}\times {\nu}) \in L^2(\partial D) \Bigr\},
\end{align*}
endowed with the norms
\begin{align*}
&\|{\Phi}\|_{\mathrm{TH}({\rm div}, \p D)}=\|{\Phi}\|_{L^2(\p D)}+\|\nabla_{\p D}\cdot {\Phi}\|_{L^2(\p D)}, \\
&\|{\Phi}\|_{\mathrm{TH}({\rm curl}, \p D)}=\|{\Phi}\|_{L^2(\p D)}+\|\nabla_{\p D}\cdot({\Phi}\times \nu)\|_{L^2(\p D)},
\end{align*}
respectively.

For a vectorial density $\Phi \in \mathrm{TH}(\mbox{div}, \p D)$,  the vectorial single layer potential is given by
\beq\label{defA}
\ds\mathcal{A}_D^{k}[\Phi](\Bx) := \int_{\p D} \Gamma_{k}(\Bx-\By)
\Phi(\By) ~\mathrm{d} s_\By, \quad \Bx \in \RR^3\setminus\p B.
\eeq
It is known that $\nabla\times\mathcal{A}_D^{k}$ satisfies the following jump formula
\begin{equation}\label{jumpM}
\nu \times \nabla \times \mathcal{A}_D^{k}[\Phi]\big\vert_\pm = \mp \frac{\Phi}{2} + \mathcal{M}_D^{k}[\Phi] \quad \mbox{on}\,  \p D,
\end{equation}
where \begin{equation*}
 \quad \nu \times \nabla \times \mathcal{A}_D^{k}[\Phi]\big\vert_\pm (\Bx)= \lim_{t\rightarrow 0^+} \nu \times \nabla \times \mathcal{A}_D^{k}[\Phi] (\Bx\pm t \nu),~~\forall \Bx\in \p D,
\end{equation*}
and
\beq\label{Mk}
\mathcal{M}^{k}_D[\Phi](\Bx)= \mbox{p.v.}\quad\nu  \times \nabla \times \int_{\p D} \Gamma_{k}(\Bx-\By) \Phi(\By) ~\mathrm{d} s_\By.
\eeq
We also define $\mathcal{L}^k_D: \mathrm{TH}({\rm div}, \p D) \rightarrow \mathrm{TH}({\rm div}, \p D)$ by
\beq\label{Lk}
\mathcal{L}^k_D[\Phi](\Bx):= \nu_\Bx  \times \nabla\times\nabla\times \mathcal{A}_D^k[\Phi](\Bx)=\nu_\Bx  \times \big(k^2\mathcal{A}_D^k[\Phi](\Bx)
+\nabla\nabla\cdot\mathcal{A}_D^k[\Phi](\Bx)\big).
\eeq

\section{The imaging scheme and its theoretical justification}\label{sec:main}

In this section, we develop the imaging scheme for locating multiple magnetised anomalies. The core is a novel and effective indicating functional. We shall first give its construction as well as use it to derive the imaging scheme. Then we shall conduct a rigorous analysis to justify the desired indicating behaviours of the imaging functional. In order to ease the exposition, we shall mainly consider the case with a single magnetised anomaly (i.e. $l_0=1$) and then remark the extension to the case with multiple but well-separated anomalies. Moreover, throughout the rest of the paper, we assume that $\Gamma$ in \eqref{eq:geom3} is an open portion of a central sphere $S_{R_0}$ with a sufficiently large radius $R_0$ that encloses $\Omega$. This assumption shall be needed in our subsequent study since we will manipulate the data on the sphere. On the other hand, it is not an essential ingredient in our study since one can reduce the inverse problem \eqref{eq:geom3} to an equivalent one by using the so-called far-field patterns, which are defined on the unit sphere as follows. From the following asymptotic expansions, one can obtain the far field patterns $\mathbf{H}_0^\infty $ and $\mathbf{H}^\infty$  (cf. \cite{CK,Ned}):
\begin{equation}\label{eq:farf1}
\begin{split}
\mathbf{H}_0(\Bx;\omega)=\frac{e^{\mathrm{i}k_0\|\Bx\|}}{\|\Bx\|} \mathbf{H}_0^\infty(\hat{\Bx};\omega)+\mathcal{O}\left(\frac{1}{\|\Bx\|^2} \right),\\
\mathbf{H}(\Bx;\omega)=\frac{e^{\mathrm{i}k_0\|\Bx\|}}{\|\Bx\|} \mathbf{H}^\infty(\hat{\Bx};\omega)+\mathcal{O}\left(\frac{1}{\|\Bx\|^2} \right),\\
\end{split}
\quad \quad ~~~~\|\Bx\|\rightarrow+\infty,
\end{equation}
where $k_0:=\omega\sqrt{\varepsilon_0\mu_0}$. By the Rellich theorem \cite{CK}, we know that the inverse problem \eqref{eq:geom3} is equivalent to the following one,
\begin{equation}\label{eq:geom2}
\mathbf{H}^\infty(\hat{\Bx};\omega)-\mathbf{H}_0^\infty(\hat{\Bx};\omega)\big|_{(\hat{\Bx},\omega)\in \hat{\Gamma}\times\mathbb{R}_+}\longrightarrow \bigcup_{l=1}^{l_0} B_l,
\end{equation}
where $\hat{\Gamma}:=\{\hat{\Bx}\in\mathbb{S}^2; \hat{\Bx}=\frac{\Bx}{\|\Bx\|},\ \Bx\in\Gamma\}\subset\mathbb{S}^2$ with $\mathbb{S}^2$ denoting the unit sphere. It can seen that all our subsequent results hold for \eqref{eq:geom2} with straightforward modifications. Hence, we stick to considering \eqref{eq:geom3} by assuming that $\Gamma$ is an open portion of $S_{R_0}$.

\subsection{Indicator functional and locating algorithm}
\label{subsec:ind}

We first consider the case with the full-aperture measurement data, namely the date on the whole sphere $S_{R_0}$. That is, we take $\Gamma=S_{R_0}$ in \eqref{eq:geom3}. We shall discuss the extensions to the limited-aperture case in Section~\ref{sect:limitedaperture}.  

We first construct an indicator functional as follows,
\begin{align}\label{eq:maininid01}
 I(\Bz):=\frac{2\|\mathbf P\|}{3R_0\|\mathbf{Q}-\mathbf Q_{\Bz}\|},
\end{align}
where $\Bz \in \Omega_s$ denotes the sampling point, and $\mathbf{P}, \mathbf{Q}$ and $\mathbf{Q}_{\mathbf{z}}$ are defined in terms of the measurement data. Next, we introduce each ingredient in \eqref{eq:maininid01}.  The vector $\mathbf{Q}\in \mathbb{C}^5$ is defined by
\begin{equation}\label{eq:Q}
\mathbf{Q}:=\int_{\mathbb{S}^{2}} (\hat{\mathbf{x}}\cdot \tilde{\mathbf H}(R_0\hat{\mathbf{x}})) \overline{\mathbf{Y}_2(\hat{\mathbf{x}})}  ~\mathrm{d} s_{\hat{\mathbf{x}}},
\end{equation}
which signifies a projection vector of $\hat{\mathbf{x}}\cdot \tilde{\mathbf H}(\mathbf{x})$ on the vector subspace
$$\mathbf{Y}_2=(Y_2^{-2}, Y_2^{-1}, Y_2^0, Y_2^1, Y_2^2)^T,$$
where $Y_n^m(\hat{\mathbf{x}}), m=-n,...,n, n=0,1,2,...$, denote the spherical harmonics. The vector $\mathbf{Q}_{\Bz}\in \mathbb{C}^5$ also signifies projection vector of $\hat{\mathbf{x}}\cdot \tilde{\mathbf H}(\mathbf{x})$ on $\mathbf{Y}_2$ with the location $\Bz_1$ replaced by a sampling point $\Bz$. To calculate $\mathbf{Q}_{\Bz}$, the vector $\mathbf P\in \mathbb{C}^3$ needs to be calculated, which is given by
\beq\label{eq:definP}
\mathbf{P}:=\int_{\mathbb{S}^{2}} (\hat{\mathbf{x}}\cdot \tilde{\mathbf H}(R_0\hat{\mathbf{x}}))\hat{\mathbf{x}} ds_{\hat{\mathbf{x}}}.
\eeq
The vector $\mathbf{P} $ is used to estimate the term $\mathbf{P}_1\mathbf{H}^{(0)}_0(\mathbf{z}_1)$ in \eqref{eq:lemag0101} (for more details, see Theorem \ref{th:P}).
Combined with our measurements, the corresponding projection coefficients of $\hat{\mathbf{x}}\cdot \tilde{\mathbf H}(R_0\hat{\mathbf{x}}) $ with the location $\Bz_1$ replaced by the sampling point $\Bz$ are given by
\begin{equation}\label{eq:Qz}
\mathbf Q_{\Bz}:=\frac{2}{3R_0}\mathbf T\mathbf{P},
\end{equation}
where $\mathbf T=(\widetilde{\mathbf T}_{-2},\widetilde{\mathbf T}_{-1},\widetilde{\mathbf T}_{0}, \widetilde{\mathbf T}_{1}, \widetilde{\mathbf T}_{2})^T
\in\mathbb C^{5\times 3}$, and
\begin{equation}\label{eq:T}
\widetilde{\mathbf T}_{m}(\mathbf z)=\sum_{h=-1}^{1}{\overline{Y_{1}^h(\hat{\mathbf{z}}})\|\bZ\|}
\int_{\mathbb{S}^{2}} \big(2Y_{1}^h(\hat{\mathbf{x}})\hat{\mathbf{x}}-\nabla_SY_{1}^{h}(\hat{\mathbf{x}})\big)\overline{Y_2^m(\hat{\mathbf{x}})} ~\mathrm{d} s_{\hat{\mathbf{x}}},
~~m=-2,...,2,
\end{equation}
where $\hat{\mathbf z}:=\mathbf z/\|\mathbf z\|$. Note that the vectors $\mathbf Q_{\Bz}$ and $\widetilde{\mathbf T}_{m}$, $m=1,2, 3, 4,5$ are presented in an explicit way. We shall explain the derivation of these vectors in Theorem \ref{th:T}.

\begin{rem}
 Since the spherical harmonics $Y_n^m, m=-n,...,n, n=0,1,2,...$, form a complete orthonormal system in $L^2(\mathbb{S}^{2})$, the expansion coefficients  of $\hat{\mathbf{x}}\cdot \tilde{\mathbf H}(\mathbf{x})$ with the location $\Bz_1$ replaced by the sampling point $\Bz$  are equivalent to that of $\hat{\mathbf{x}}\cdot \tilde{\mathbf H}(\mathbf{x})$  if and only if $\Bz=\Bz_1$. We need to choose a suitable subset of spherical harmonics  since space $L^2(\mathbb S^2)$ is  infinitely dimensional.
Here we choose a suitable subspace, that is $\{Y_2^{-2}, Y_2^{-1}, Y_2^0, Y_2^1, Y_2^2\}$. It is worth mentioning that for the subspace $\{Y_1^{-1}, Y_1^0, Y_1^1\}$,  it contains no information on the position of the anomaly.
\end{rem}

Next, we investigate the asymptotic property of the indicator functional $I(\Bz)$ to provide a theoretical basis of our algorithm. We have the following result:
\begin{thm}\label{thm:main3}
It holds that $I(\Bz)=\Ocal(\|\Bz_1-\Bz\|^{-1})$ with $\Bz \rightarrow \Bz_1$. Moreover, one has
\begin{equation}\label{eq:ib1}
I(\Bz)\leq \sqrt{\frac{4\pi}{3}} \bigg({\|{\mathbf z}_1-{\mathbf z}\|  \min_{\|\mathbf y\|=1}\max_{-2\leq m\leq 2}|\mathbf y^T\mathbf D_m \hat{\mathbf P}|}\bigg)^{-1}
\end{equation}
with
$\|{\mathbf z}_1-{\mathbf z}\| \gg 0$, where $\hat{\mathbf P}:=\mathbf P/\|\mathbf P\|$.
\end{thm}

From \eqref{eq:ib1}, one can readily see that the indicator functional blows up when the sampling point is  sufficiently close to the location of the anomaly and is bounded when the sampling point is located away from it. Based on the indicator functional \eqnref{eq:maininid01} and its behavior, we design the algorithm for locating the magnetized anomaly in Algorithm~\ref{alg:one}.

\begin{varalgorithm}{3.1}
\caption{Imaging scheme for locating magnetized anomalies }
\label{alg:one}
\begin{algorithmic}
\State{\textbf{step 1} \quad Set $N_1$, $N_2$ and $N_3$. Using \eqnref{eq:min} to extend $\tilde{\mathbf H}(\mathbf{x})$ from $\Gamma$ to $S_{R_0}$. }

\State{\textbf{step 2} \quad Calculate the integral $\mathbf{Q}=\int_{\mathbb{S}^{2}} (\hat{\mathbf{x}}\cdot \tilde{\mathbf H}(R_0\hat{\mathbf{x}}))\overline{\mathbf{Y}_2(\hat{\mathbf{x}})}  ~\mathrm{d} s_{\hat{\mathbf{x}}}.$}

\State{\textbf{step 3} \quad Calculate the integral $\mathbf{P}=\int_{\mathbb{S}^{2}} (\hat{\mathbf{x}}\cdot \tilde{\mathbf H}(R_0\hat{\mathbf{x}}))\hat{\mathbf{x}} ~\mathrm{d} s_{\hat{\mathbf{x}}}.$  }

\State{\textbf{step 4} \quad  For any sampling points $\mathbf{z}\in \Omega_s$ calculate the integrals $$\widetilde{\mathbf T}_{m}=\sum_{h=-1}^{1}{\overline{Y_{1}^h(\hat{\mathbf{z}}})\|\bZ\|}
\int_{\mathbb{S}^{2}} \big(2Y_{1}^h(\hat{\mathbf{x}})\hat{\mathbf{x}}-\nabla_SY_{1}^{h}(\hat{\mathbf{x}})\big)\overline{Y_2^m(\hat{\mathbf{x}})} ~\mathrm{d} s_{\hat{\mathbf{x}}}.$$ }

\State{\textbf{step 5} \quad Calculate the projection vector
$$\mathbf Q_{\Bz}=\frac{2}{3R_0}\mathbf T\mathbf{P}, \quad \mathbf T=(\widetilde{\mathbf T}_{-2},\widetilde{\mathbf T}_{-1},\widetilde{\mathbf T}_{0}, \widetilde{\mathbf T}_{1}, \widetilde{\mathbf T}_{2})^T.$$}

\State{\textbf{step 6} \quad Calculate the imaging functional
$$ I(\Bz)=\frac{2\|\mathbf P\|}{3R_0\|\mathbf{Q}-\mathbf Q_{\Bz}\|}.$$  }

\State{\textbf{step 7} \quad Plot the value of the normalized imaging functional in the spatial sampling grid
$$ \mathcal{I}(\Bz) =  \frac{I(\Bz) }{\max_{\Bz} I(\Bz) } .$$}
\end{algorithmic}
\end{varalgorithm}

It is noted that in our algorithm, we only use the data of the geomagnetic fields on the surface $\Gamma$ as well as certain inner product operations to construct the imaging functional for determining the location of the magnetized anomaly. Hence, it is computationally efficient and moreover it is robust against noises that are possibly involved. 

\subsection{Proof of Theorem~\ref{thm:main3}}
\label{subsec:proof}

In this subsection, we provide the proof of the main result in Section \ref{subsec:ind}, i.e. Theorem~\ref{thm:main3}. Before doing this, we first briefly summarize some results on the integral representations and asymptotic behaviors of the geomagnetic fields in our study, which shall be needed in justifying the indicating behavior in Theorem~\ref{thm:main3}. Moreover, we would like to emphasise that in Theorem~\ref{thm:main3}, we only consider the case with a single anomaly, namely $l_0=1$. Nevertheless, as remarked earlier, we shall consider the extension to the case with multiple anomalies in Section~\ref{sect:limitedaperture}. To that end, for some of results in this section, we also consider multiple anomalies if the corresponding results are true for such a case. 

First, in the absence of any anomaly in the Earth's shell, the geomagnetic field $\bH_0$ has the following integral ansatz:
\beq\label{eq:repre02}
\bH_0=\left\{\begin{split}
&-\mathrm{i}\omega^{-1}\Big(\nabla\times\nabla\times\Acal_{\Omega_c}^{k_0}[\Phi_C]+\nabla\times\nabla\times\Acal_{\Omega}^{k_0}[\Phi_S]+\\
& \hspace*{4cm}\omega^2\varepsilon_0\nabla\times\Acal_{\Omega}^{k_0}[\Psi_S]\Big) \ \mbox{in} \ \ \RR^3\setminus\overline{\Omega}, \\
&-\mathrm{i}\omega^{-1}\Big(\nabla\times\nabla\times\Acal_{\Omega_c}^{k_s}[\Phi_C]+\nabla\times\nabla\times\Acal_{\Omega}^{k_s}[\Phi_S]+\\
&\hspace*{4cm}\omega^2\varepsilon_s\nabla \times\Acal_{\Omega}^{k_s}[\Psi_S]\Big)\ \mbox{in} \ \ \Omega_s,
\end{split}
\right.
\eeq
where $(\Phi_S,\Psi_S, \Phi_C)\in \mathrm{TH}(\mbox{div},\p \Omega)\times \mathrm{TH}(\mbox{div},\p \Omega)\times \mathrm{TH}(\mbox{div},\p \Omega_c)$. The surface density functions $(\Phi_S,\Psi_S, \Phi_C)$ are (uniquely) determined by the boundary conditions on $\p\Omega_s$ as well as the transmission conditions across $\p \Omega$.  Let $\bH^{(0)}_0$  be the leading-order term  of $\bH_0$  with respect to frequency $\omega$, then it holds that
\beq\label{eq:le0101}
\bH_0=\bH_0^{(0)}+ \Ocal(\omega):=\nabla\Scal_{\Omega_c}^{0}\Big(\frac{I}{2}+(\Kcal_{\Omega_c}^{0})^*\Big)^{-1}[\nu\cdot\bH_0|_{\p \Omega_c}^{+}]+\Ocal(\omega)\ \ \rm{in} \ \ \RR^3\setminus\overline{\Omega_c}.
\eeq
Similarly, the geomagnetic field $\bH$ with the presence of the anomalies has following integral representation:
\beq\label{eq:solrepre02}
\bH=\left\{
\begin{split}
\hat\bH_0-\mathrm{i}\omega^{-1}\nabla\times\Big((\omega^2\varepsilon_0\Acal_{\Omega}^{k_0}[\Psi_0]+\nabla\times\Acal_{\Omega}^{k_0}[\Phi_0])&\\
+\sum_{l'=1}^{l_0}(\omega^2\varepsilon_0\Acal_{B_{l'}}^{k_0}[\Psi_{l'}]+\nabla\times\Acal_{B_{l'}}^{k_0}[\Phi_{l'}])\Big) &\quad \mbox{in} \ \ \RR^3\setminus\overline\Omega, \\
\hat\bH_0-\mathrm{i}\omega^{-1}\nabla\times\Big((\omega^2\varepsilon_s\Acal_{\Omega}^{k_s}[\Psi_0]+\nabla\times\Acal_{\Omega}^{k_s}[\Phi_0])&\\
+\sum_{l'=1}^{l_0}(\omega^2\varepsilon_s\Acal_{B_{l'}}^{k_s}[\Psi_{l'}]+\nabla\times\Acal_{B_{l'}}^{k_s}[\Phi_{l'}])\Big) &\quad \mbox{in} \ \ \tilde{B}, \\
-\mathrm{i}\omega^{-1}\nabla\times\Big((\omega^2\gamma_l\Acal_{\Omega}^{\varsigma_l}[\Psi_0]+\nabla\times\Acal_{\Omega}^{\varsigma_l}[\Phi_0])&\\
+\sum_{l'=1}^{l_0}(\omega^2\gamma_{l}\Acal_{B_{l'}}^{\varsigma_l}[\Psi_{l'}]+\nabla\times\Acal_{B_{l'}}^{\varsigma_l}[\Phi_{l'}])\Big) &\quad \mbox{in} \ \ B_l,
\end{split}
\right.
\eeq
where $\tilde{B}:=\Omega_s\setminus\overline{\bigcup_{l'=1}^{l_0}B_{l'}}$, $k_s:=\omega\sqrt{\mu_0\varepsilon_s}$, $(\Phi_0, \Psi_0)\in \rm{TH}({\rm div},\p \Omega)\times \rm{TH}({\rm div},\p \Omega)$ and $(\Phi_l, \Psi_l)\in \mathrm{TH}({\rm div}, \p B_l)\times\mathrm{TH}({\rm div}, \p B_l)$, $l=1, 2, \ldots, l_0$. In \eqref{eq:solrepre02}, the surface density functions are (uniquely) determined by the boundary conditions on $\p \Omega_s$ as well as the transmission conditions across $\p \Omega$ and $\p B_l$, $l=1, 2, \ldots, l_0$. Moreover, the field $\hat\bH_0$ has the following expansion:
 \beq\label{eq:leH1repre01}
\begin{split}
\bH_0&=\hat\bH_0+(\varepsilon_0-\varepsilon_s)\varepsilon_s^{-1}\sum_{l'=1}^{l_0}\nabla\Scal_{B_{l'}}^0[\nu_{l'}\cdot\hat\bH_0]+\Ocal(\omega) \ \ \mbox{in} \ \ (\RR^3\setminus\overline\Omega)\bigcup \tilde{B}.
\end{split}
 \eeq
Then $\bH$ equipped with the leading-order term $\bH^{(0)}$ has the following asymptotic expansions with respect to the frequency $\omega$ as $\omega\rightarrow 0^+$:
\beq\label{eq:th0101}
\bH=\bH^{(0)}+\Ocal(\omega):=\left\{
\begin{split}
\hat\bH_0&-\varepsilon_0\nabla\times\Acal_\Omega^0[\Xi]+\nabla\Scal_{\Omega}^0[\Theta]\\
&+\sum_{l'=1}^{l_0}\Big(\varepsilon_0\nabla\times\Acal_{B_{l'}}^{0}[\Psi_{l'}^{(0)}]-\mu_0\nabla\Scal_{B_{l'}}^{0}[\Pi_{l'}]\Big)+\Ocal(\omega) \quad \mbox{in} \ \ \RR^3\setminus\overline{\Omega}, \\
\hat\bH_0&-\varepsilon_s\nabla\times\Acal_\Omega^0[\Xi]+\nabla\Scal_{\Omega}^0[\Theta]\\
&+\sum_{l'=1}^{l_0}\Big(\varepsilon_s\nabla\times\Acal_{B_{l'}}^{0}[\Psi_{l'}^{(0)}]-\mu_0\nabla\Scal_{B_{l'}}^{0}[\Pi_{l'}]\Big)+\Ocal(\omega) \quad \mbox{in} \ \ \tilde{B}, \\
&-\gamma_l\nabla\times\Acal_\Omega^0[\Xi]-\nabla\Scal_{\Omega}^0[\Theta]\\
&+\sum_{l'=1}^{l_0}\Big(\gamma_l\nabla\times\Acal_{B_{l'}}^{0}[\Psi_{l'}^{(0)}]-\mu_0\nabla\Scal_{B_{l'}}^{0}[\Pi_{l'}]\Big)+\Ocal(\omega) \quad  \mbox{in} \ B_l,
\end{split}
\right.
\eeq
where $\Xi, \Theta\in {\rm TH}({\rm div}, \p \Omega)$ and $\Psi_l^{(0)}\in {\rm TH}({\rm div}, \p B_l)$, $l=1, 2,\ldots, l_0$.
$\Pi_l\in L^2(\p B_l)$, $l=1, 2,\ldots, l_0$ are defined by
\beq\label{eq:th0103}
\begin{split}
\Pi_l=&\Big((\mathbb{J}_B^{\mu})^{-1}\big[(\frac{\nu_1\cdot\hat\bH_0}{\mu_1-\mu_0}, \frac{\nu_2\cdot\hat\bH_0}{\mu_2-\mu_0}, \ldots, \frac{\nu_{l_0}\cdot\hat\bH_0}{\mu_{l_0}-\mu_0})^T\big]\Big)_l \\
&-\Big((\mathbb{J}_B^{\mu})^{-1}\big[(\frac{\omega\gamma_1\mu_1\nu_1\cdot\mathbf{C}}{\mu_1-\mu_0}, \frac{\omega\gamma_2\mu_2\nu_2\cdot\mathbf{C}}{\mu_2-\mu_0}, \ldots, \frac{\omega\gamma_{l_1}\mu_{l_0}\nu_{l_0}\cdot\mathbf{C}}{\mu_{l_0}-\mu_0})^T\big]\Big)_l,
\end{split}
\eeq
where  $\gamma_l:=\Ge_l+i\Gs_l/\omega$, $\mathbb{J}_B^{\mu}$ is defined on $L^2(\p B_1)\times L^2(\p B_2)\times\cdots\times L^2(\p B_{l_0})$ and given by
\beq\label{eq:app10}
\mathbb{J}_B^{\mu}:=\left(
\begin{array}{cccc}
\lambda_{\mu_1}I & 0 & \cdots & 0\\
0 & \lambda_{\mu_2}I & \cdots & 0\\
\vdots & \vdots &\ddots &\vdots \\
0 & 0 & \cdots & \lambda_{\mu_{l_0}}I
\end{array}
\right)-
\left(
\begin{array}{cccc}
(\Kcal_{B_1}^{0})^* & \nu_1\cdot\nabla\Scal_{B_2}^{0} & \cdots & \nu_1\cdot\nabla\Scal_{B_{l_0}}^{0} \\
\nu_2\cdot\nabla\Scal_{B_1}^{0} & (\Kcal_{B_2}^{0})^* & \cdots & \nu_2\cdot\nabla\Scal_{B_{l_0}}^{0}\\
\vdots & \vdots &\ddots &\vdots \\
\nu_{l_0}\cdot\nabla\Scal_{B_1}^{0} & \nu_{l_0}\cdot\nabla\Scal_{B_2}^{0} & \cdots & (\Kcal_{B_{l_0}}^{0})^*
\end{array}
\right),
\eeq
 and $\mathbf{C}$ is defined by
\beq\label{eq:th01add01}
\mathbf{C}:=\nabla\times\Acal_\Omega^0\Big(\Gl_{\varepsilon} I +\Mcal_{\Omega}^0\Big)^{-1}\sum_{l'=1}^{l_0}\Mcal_{\Omega, B_{l'}}^0[\Psi_{l'}^{(0)}]
-\nabla\times\sum_{l'=1}^{l_0}\Acal_{B_{l'}}^0[\Psi_{l'}^{(0)}].
\eeq
The parameters $\Gl_{\mu_l}$ and $\Gl_{\gamma_l}$ are defined by
\beq\label{eq:defmulam}
\lambda_{\mu_l}:=\frac{\mu_l+\mu_0}{2(\mu_l-\mu_0)}, \quad \lambda_{\gamma_l}:=\frac{\gamma_l+\varepsilon_s}{2(\gamma_l-\varepsilon_s)}, \quad l=1, 2, \ldots, l_0.
\eeq

For more details on the above results, we refer to \cite{DLL:19}. By gazing at the geomagnetic field $\bH$ as a perturbed field of $\bH_0$, with the aid of the integral representation of the geomagnetic fields $\bH_0$ and $\bH$, the mapping properties of the boundary integral operators as well as the asymptotic analysis with respect to the size $\delta$, we can arrive at the following result, which is summarized from \cite{DLL:19}. 
\begin{lem}\label{le:maggra01}
Suppose $B_l$, $l=1, 2 ,\ldots, l_0$ are defined in \eqnref{eq:permeab02} with $\delta\in\mathbb{R}_+$ sufficiently small. Then  there holds the following asymptotic expansion:
\beq\label{eq:lemag0101}
\bH^{(0)}(\Bx)=
\bH_0^{(0)}(\Bx)+\delta^3\sum_{l=1}^{l_0}\nabla\big(\nabla\Gamma_0(\Bx-\Bz_l)^T\mathbf{P}_l\bH_0^{(0)}(\Bz_l)\big)+\Ocal(\delta^4),\quad \Bx\in \RR^3\setminus\Omega,
\eeq
where $\mathbf{P}_l$ is a $3\times 3$ matrix defined by
\beq\label{eq:matM01}
\mathbf{P}_l:=\mu_0\mathbf{M}_l-\varepsilon_0\mathbf{B}_l-\mathbf{P}_0, l=1, 2, \ldots, l_0.
\eeq
$\mathbf{P}_0$ is defined by
\beq\label{eq:leasymdefP0}
\mathbf{P}_0:=\int_{\p \Omega}\tilde\By(\lambda_{\varepsilon} I-(\Kcal_{\Omega}^{0})^*)^{-1}[\nu_l]ds_{\tilde\By},
\eeq
 and the polarization tensors $\mathbf{B}_l$ and $\mathbf{M}_l$ are $3\times 3$ matrices defined by
\beq\label{eq:leasym02}
\mathbf{B}_l:=\frac{1}{\gamma_l-\varepsilon_s}\frac{\varepsilon_s}{\varepsilon_s-\varepsilon_0}\int_{\p \Omega}\tilde\By(\lambda_{\gamma_l} I+(\Kcal_{\Omega}^{0})^*)^{-1}(\lambda_{\varepsilon} I-(\Kcal_{\Omega}^{0})^*)^{-1}[\nu_l] ~\mathrm{d} s_{\tilde\By},
\eeq
and
\beq\label{eq:leasym03}
\mathbf{M}_l:=\frac{1}{\mu_l-\mu_0}\frac{\varepsilon_s}{\varepsilon_s-\varepsilon_0}\int_{\p \Omega}\tilde\By(\lambda_{\mu_l} I-(\Kcal_{\Omega}^{0})^*)^{-1}(\lambda_{\varepsilon} I-(\Kcal_{\Omega}^{0})^*)^{-1}[\nu_l] ~\mathrm{d} s_{\tilde\By},
\eeq
respectively, $l=1, 2, \ldots, l_0$, where
$$
\Gl_{\varepsilon}:=\frac{\varepsilon_s+\varepsilon_0}{2(\varepsilon_s-\varepsilon_0)}.
$$
\end{lem}

This is a key result which characterizes the variation of the geomagnetic field due to the presence of the anomalies. In particular, we can see that the change of the geomagnetic field caused by the presence of magnetized anomalies is a gradient field. 

We proceed to consider the proof of Theorem~\ref{thm:main3} and switch our discussion to the case with $l_0=1$ for simplicity. We recall the following property concerning the fundamental solution $\Gamma_0$. 

\begin{lem}\cite{DLL:19}\label{le:04}
Let $\bZ\in \RR^3$ be fixed,  $\Bx\in S_R$, and suppose $\|\bZ\|<R$. Then there holds the following expansion:
\beq\label{eq:leasm02}
\nabla \Gamma_0(\Bx-\bZ)=\sum_{n=0}^{\infty}\sum_{m=-n}^{n}
\frac{(n+1)Y_n^m(\hat\Bx)\hat\Bx-\nabla_SY_n^m(\hat\Bx)}{(2n+1)R^{n+2}}\overline{Y_n^m(\hat{\bZ})}\|\bZ\|^{n},
\eeq
\end{lem}

Note that the matrix $\mathbf{P}_1$ in \eqref{eq:matM01} has a rather complicated expression and we do not have any measurement data on the shell. Hence, we need to estimate  $\mathbf{P}_1\mathbf{H}^{(0)}_0(\mathbf{z}_1)$ using the measurement data. Using Lemma \ref{le:04}, one can obtain $\mathbf{P}_1\mathbf{H}^{(0)}_0(\mathbf{z}_1)$ by using the full data $\tilde{\mathbf H}(\mathbf x)$ on $S_{R_0}$.

\begin{thm}\label{th:P}
Suppose $\mathbf{x}\in S_{R_0}$ is outside the surface of the Earth. Then there holds the following asymptotic expansion:
\begin{equation}\label{eq:PH0}
\delta^3\mathbf{P}_1\bH_0^{(0)}(\Bz_1)=-6R_0^3 \mathbf{P}+\Ocal(\delta^4)+\Ocal(\omega),
\end{equation}
where $\mathbf{P}$ is given in (\ref{eq:definP}).
\end{thm}
\begin{proof}
By using \eqref{eq:lemag0101}, \eqref{eq:le0101} and \eqref{eq:th0101}, we can obtain
\begin{equation}\label{eq:P}
\begin{split}
\mathbf{P}:&=\int_{\mathbb{S}^{2}} (\hat{\mathbf{x}}\cdot \tilde{\mathbf H}(R_0\hat{\mathbf{x}}))\hat{\mathbf{x}} ~\mathrm{d} s_{\hat{\mathbf{x}}} \\
&=\int_{\mathbb{S}^{2}} \big(\hat{\mathbf{x}}\cdot\delta^3\nabla\big(\nabla\Gamma_0(R_0\hat{\mathbf{x}}-\Bz_1)^T\mathbf{P}_1\bH_0^{(0)}(\Bz_1)\big)\big)\hat{\mathbf{x}}  ~\mathrm{d} s_{\hat{\mathbf{x}}}+\Ocal(\delta^4)+\Ocal(\omega)\\
&=\delta^3\int_{\mathbb{S}^{2}} \big(\frac{\p}{\p r}\big(\nabla\Gamma_0(R_0\hat{\mathbf{x}}-\Bz_1)^T\mathbf{P}_1\bH_0^{(0)}(\Bz_1)\big)\big)\hat{\mathbf{x}} ~\mathrm{d} s_{\hat{\mathbf{x}}}+\Ocal(\delta^4)+\Ocal(\omega)\\
&=-\frac{\delta^3}{6R_0^3}\mathbf{P}_1\bH_0^{(0)}(\Bz_1)+\Ocal(\delta^4)+\Ocal(\omega),
\end{split}
\end{equation}
where the last equality can be obtained by using the symmetry of spherical integral and the fact that
\begin{equation*}
\int_{\mathbb{S}^{2}} \Ncal_n^mY_{n'}^{m'} ~\mathrm{d} s=\mathbf{0}, n\neq n'.
\end{equation*}
This readily completes the proof.
\end{proof}

By using the asymptotic behavior of the fundamental solution, one can also see that the term $\hat{\mathbf{x}}\cdot \tilde{\mathbf H}(\mathbf{x})$ has orthogonal expansion with respect to the spherical harmonics outside the surface of the Earth.
\begin{thm} \label{th:T}
Suppose $\|\mathbf{x}\|=R$. Then there holds the following orthogonal expansion with respect to spherical harmonics $Y_n^m$:
\begin{equation}\label{eq:XI0Y}
\hat{\mathbf{x}}\cdot \tilde{\mathbf H}(\mathbf{x})\approx-6R_0^3\sum_{n=1}^{\infty}\sum_{m=-n}^{n} \mathbf T_{n,m}^T\mathbf{P}Y_n^m(\hat\Bx),~~
\mathbf{x}\in\RR^3\setminus\overline\Omega,
\end{equation}
where
\begin{equation}\label{eq:T}
\mathbf T_{n,m}=\sum_{h=-(n-1)}^{n-1}\frac{\overline{Y_{n-1}^h(\hat{\mathbf{z}}_1})\|\bZ_1\|^{n-1}}{(2n-1)(-n-1)R^{n+2}}
\int_{\mathbb{S}^{2}} \big(nY_{n-1}^h(\hat{\mathbf{x}})\hat{\mathbf{x}}-\nabla_SY_{n-1}^{h}(\hat{\mathbf{x}})\big)\overline{Y_n^m(\hat{\mathbf{x}}}) ~\mathrm{d} s_{\hat{\mathbf{x}}}, n\geq 1.
\end{equation}
\end{thm}
\begin{proof}
By using \eqref{eq:lemag0101}, \eqref{eq:le0101}, \eqref{eq:th0101} and \eqref{eq:PH0}, we can obtain for $\mathbf{x}\in\RR^3\setminus\overline\Omega$ that
\begin{equation}
\begin{split}
&~~~~\int_{\mathbb{S}^{2}} (\hat{\mathbf{x}}\cdot \tilde{\mathbf H}(R\hat{\mathbf{x}}))\overline{Y_n^m(\hat{\mathbf{x}})} ~\mathrm{d} s_{\hat{\mathbf{x}}} \\
&=\int_{\mathbb{S}^{2}} \big(\hat{\mathbf{x}}\cdot\delta^3\nabla\big(\nabla\Gamma_0(R\hat{\mathbf{x}}-\Bz_1)^T\mathbf{P}_1\bH_0^{(0)}(\Bz_1)\big)\big)\overline{Y_n^m(\hat{\mathbf{x}})} ~\mathrm{d} s_{\hat{\mathbf{x}}}+\Ocal(\delta^4)+\Ocal(\omega)\\
&=\delta^3\int_{\mathbb{S}^{2}} \big(\frac{\p}{\p r}\big(\nabla\Gamma_0(R\hat{\mathbf{x}}-\Bz_1)^T\mathbf{P}_1\bH_0^{(0)}(\Bz_1)\big)\big)\overline{Y_n^m(\hat{\mathbf{x}})} ~\mathrm{d} s_{\hat{\mathbf{x}}}+\Ocal(\delta^4)+\Ocal(\omega)\\
&\approx -6R_0^3\int_{\mathbb{S}^{2}} \big(\frac{\p}{\p r}\big(\nabla\Gamma_0(R\hat{\mathbf{x}}-\Bz_1)^T\mathbf{P}\big)\big)\overline{Y_n^m(\hat{\mathbf{x}})} ~\mathrm{d} s_{\hat{\mathbf{x}}}\\
&=-6R_0^3\int_{\mathbb{S}^{2}} \big(\frac{\p}{\p r}\big(\nabla\Gamma_0(R\hat{\mathbf{x}}-\Bz_1)^T\overline{Y_n^m(\hat{\mathbf{x}})}\big)\big)\mathbf{P} ~\mathrm{d} s_{\hat{\mathbf{x}}}\\
&=-6R_0^3\sum_{n=0}^{\infty}\sum_{m=-n}^{n} \mathbf T_{n,m}^T\mathbf{P},\\
\end{split}
\end{equation}
where the last equality can be obtained by using Proposition $2.2$ in \cite{DLL:20}. This readily completes the proof. 
\end{proof}

Hence, the information of the unknown position $\Bz_1$ only exists in the coefficients of the orthogonal expansion of $\hat{\mathbf{x}}\cdot \tilde{\mathbf H}(\mathbf{x})$ with respect to the spherical harmonics. More specifically, it is only contained in $\mathbf{T}_n^m$. This observation suggests that we use the orthogonal expansion coefficients to construct the indicator functional in reconstructing the location of the magnetized anomaly.

To simplify the analysis of the indicator functional, we next present the matrix representation of $\Ncal_2^{m}$ with respect to the spherical harmonics of order $1$. By straightforward calculations, one can obtain the following formula:
\begin{equation}
\Ncal_2^{m}(\hat{\mathbf{x}})=-\nabla_{S}Y_{1}^{m}(\hat{\mathbf{x}}) + 2Y_{1}^{m}(\hat{\mathbf{x}}) \hat{\mathbf{x}}=\mathbf N^{(m)} \mathbf Y_2, ~~m=-1,0,1,
\end{equation}
where

\begin{eqnarray*}
\begin{split}
&\mathbf N^{(-1)}=(n_{ij}^{(-1)})=\begin{bmatrix}
\frac{3\sqrt{5}}{5} & 0 &  -\frac{\sqrt{30}}{10} & 0 & 0 \\
\frac{3\sqrt{5}}{5}\mathrm{i} & 0 &  \frac{\sqrt{30}}{10}\mathrm{i} & 0 & 0 \\
0 & \frac{3\sqrt{5}}{5} &  0 & 0 & 0 \\
\end{bmatrix},\\
&\mathbf N^{(0)}=(n_{ij}^{(0)})=\begin{bmatrix}
0 & \frac{3\sqrt{10}}{10} &  0 & -\frac{3\sqrt{10}}{10} & 0 \\
0 & \frac{3\sqrt{10}}{10}\mathrm{i} &  0 & \frac{3\sqrt{10}}{10}\mathrm{i} & 0 \\
0 & 0 &  \frac{2\sqrt{15}}{5} & 0 & 0 \\
\end{bmatrix},\\
&\mathbf N^{(1)}=(n_{ij}^{(1)})=\begin{bmatrix}
0 & 0 &\frac{\sqrt{30}}{10} & 0 &  -\frac{3\sqrt{5}}{5}  \\
0 & 0 &\frac{\sqrt{30}}{10}\mathrm{i} & 0 &  \frac{3\sqrt{5}}{5}\mathrm{i}  \\
0 & 0 &  0 & \frac{3\sqrt{5}}{5} & 0 \\
\end{bmatrix}.\\
\end{split}
\end{eqnarray*}
Therefore, by combing the above results, we can obtain that
\begin{equation}
\widetilde{\mathbf T}_m^T \mathbf P=\widetilde{\mathbf z}^T \mathbf D_m \mathbf P,
\end{equation}
where
$\mathbf D_m=(d_{ij}^{(m)})\in \mathbb{C}^{3\times 3}$, $d_{ij}^{(m)}=n_{j,m+3}^{(i-2)}$, $\widetilde{\mathbf z}=(\widetilde{z}_1,\widetilde{z}_2,\widetilde{z}_3)^T \in \mathbb{C}^3$ with $\widetilde{z}_i= {\overline{Y_{1}^{i-2}(\hat{\mathbf{z}}})\|\bZ\|}$, $m=-2,-1,0,1,2;~ i,j=1,2,3$.
We have the following key observation:
\begin{lem}\label{le:rank}
For a fixed  nonzero vector $\mathbf a\in \mathbb C^3$, 
\begin{equation}\label{eq:rank}
\mathrm{rank} [\mathbf D_{-2}\mathbf a, \mathbf D_{-1}\mathbf a,\mathbf D_{0}\mathbf a,\mathbf D_{1}\mathbf a,\mathbf D_{2}\mathbf a]=3.
\end{equation}
\end{lem}
\begin{proof}
The statement can be verified by straightforward computations of the terms $\mathbf D_m\mathbf a, m=-2,-1,0,1,2$.
\end{proof}

Based on the above preparations, we are in a position to present the proof of Theorem \ref{thm:main3} as follows. 

\begin{proof}
With the help of the spherical harmonics of order $1$, for $m=-2,-1,0,1,2$ and $\Bz \rightarrow \Bz_1$, we have
\begin{equation*}
|\widetilde{\mathbf z}_1^T \mathbf D_m \hat{\mathbf P}-\widetilde{\mathbf z}^T \mathbf D_m \hat{\mathbf P}|\leq \sum_{i=1}^{3}|({\overline{Y_{1}^{i-2}(\hat{\mathbf{z}}_1})\|\bZ_1\|}-{\overline{Y_{1}^{i-2}(\hat{\mathbf{z}}})\|\bZ\|})(\mathbf D_m \hat{\mathbf P})_i|=\Ocal(\|\Bz_1- \Bz\|_1),
\end{equation*}
where  $\widetilde{\mathbf z}_1=(\widetilde{z}_1^{(1)},\widetilde{z}_2^{(1)},\widetilde{z}_3^{(1)})^T \in \mathbb{C}^3$ with $\widetilde{z}_i^{(1)}= {\overline{Y_{1}^{i-2}(\hat{\mathbf{z}}_1})\|\bZ_1\|}$. Hence
\begin{equation*}
\frac{3R_0\|\mathbf{Q}-\mathbf Q_{\Bz}\|}{2\|\mathbf P\|}=\Ocal(\|\Bz_1- \Bz\|_1)=\Ocal(\|\Bz_1- \Bz\|). 
\end{equation*}
That is,  $I(\Bz)=\Ocal(\|\Bz_1-\Bz\|^{-1})$ with $\Bz \rightarrow \Bz_1$.

On the other hand,
\begin{equation*}
\frac{3R_0\|\mathbf{Q}-\mathbf Q_{\Bz}\|}{2\|\mathbf P\|}\geq \max_{-2\leq m\leq 2}|(\widetilde{\mathbf z}_1-\widetilde{\mathbf z})^T \mathbf D_m \hat{\mathbf P}| \geq \|\widetilde{\mathbf z}_1-\widetilde{\mathbf z}\| \min_{\|\mathbf y\|=1}\max_{-2\leq m\leq 2}|\mathbf y^T\mathbf D_m \hat{\mathbf P}|.
\end{equation*}
By direct calculations, one can obtain that
\begin{equation*}
 \|\widetilde{\mathbf z}_1-\widetilde{\mathbf z}\|= \sqrt{\frac{3}{4\pi}} \|{\mathbf z}_1-{\mathbf z}\|. 
\end{equation*}
Moreover, using Lemma \ref{le:rank}, we have
\begin{equation*}
\min_{\|\mathbf y\|=1}\max_{-2\leq m\leq 2}|\mathbf y^T\mathbf D_m \hat{\mathbf P}|> 0.
\end{equation*}
Hence, by combining the estimates above we have for $\|{\mathbf z}_1-{\mathbf z}\| \gg 0$ that
\begin{equation*}
I(\Bz) \leq \sqrt{\frac{4\pi}{3}} \bigg({\|{\mathbf z}_1-{\mathbf z}\|  \min_{\|\mathbf y\|=1}\max_{-2\leq m\leq 2}|\mathbf y^T\mathbf D_m \hat{\mathbf P}|}\bigg)^{-1}. 
\end{equation*}

The proof is complete.
\end{proof}

\subsection{Two extensions}\label{sect:limitedaperture}

In this subsection, we consider two extensions of the results in the previous subsections to a more general and practical scenario. First, we consider the case with multiple anomalies, namely $l_0>1$ in \eqref{eq:geom3}. In such a case, we assume that the multiple anomalies $B_l$ are well separated, i.e. $\mathrm{dist}(B_l, B_{l'})\gg 1$ when $l\neq l'$. In such a case, the multiple scattering effect among different anomalous components are negligible. Hence, Theorem~\ref{thm:main3} can be directly extend to the case when $\mathbf{z}_1$ is replaced by any $\mathbf{z}_l$, $l=1,2,\ldots,l_0$. Therefore, Algorithm~\ref{alg:one} can be directly extended to recover multiple well-separated magnetized anomalies without any modification. On the other hand, as can be seen from our numerical experiments in Section~\ref{sec:experiments} in what follows, our method still works well even when the anomalous components are reasonably close. The main justification of this kind of performance is due to that we are mainly concerned with the qualitative locating of the positions of the anomalies. 

The other extension is about the limited-aperture measurement data; that is, in \eqref{eq:geom3}, the measurement data is only given on $\Gamma$ which is a proper open subset of $S_{R_0}$. There are two means to tackle such a situation. First, one can simply replace the integral domain $\mathbb{S}^2$ in \eqref{eq:maininid01}--\eqref{eq:T} in defining the imaging functional by $\Gamma$. In Section~\ref{sec:experiments}, our numerical experiments show that our scheme can still locate multiple anomalies if $\Gamma$ is not too small. The other approach is to extend the data of the geomagnetic field from the partial surface $\Gamma$ to the whole sphere, which is referred to as the data extension or the data interpolation.
To that end, we recall that the vectorial spherical harmonics of order $n$ are given by
\begin{equation}
\begin{split}
  \Ical_n^m(\hat{\mathbf{x}})=&\nabla_{S}Y_{n+1}^m(\hat{\mathbf{x}}) +(n+1) Y_{n+1}^m(\hat{\mathbf{x}}) \hat{\mathbf{x}}, \quad n\geq 0,  n+1\geq m \geq -(n+1),\\
  \Tcal_n^m(\hat{\mathbf{x}})=&\nabla_{S}Y_{n}^m(\hat{\mathbf{x}}) \times \hat{\mathbf{x}}, \quad  n\geq 1,  n\geq m \geq -n,\\
  \Ncal_n^m(\hat{\mathbf{x}})=&-\nabla_{S}Y_{n-1}^m(\hat{\mathbf{x}}) + nY_{n-1}^m(\hat{\mathbf{x}}) \hat{\mathbf{x}}, \quad n\geq1,  n-1\geq m \geq -(n-1),\\
\end{split}
\end{equation}
which form a complete orthogonal basis of $(L^2(\mathbb{S}^2))^3$. Based on this result, we can approximate $\tilde{\mathbf{H}}$ by using the vectorial spherical harmonics as follows
\begin{equation}
\tilde{\mathbf{H}}(R_0\hat{\mathbf{x}})=\big(\sum_{n=0}^{N_1}\sum_{m=-n}^{n}\alpha_n^m\Ical_n^m+\sum_{n=1}^{N_2}\sum_{m=-n}^{n}\beta_n^m\Tcal_n^m
+\sum_{n=1}^{N_3}\sum_{m=-n}^{n}\rho_n^m\Ncal_n^m\big)(\hat{\mathbf{x}}), \hat\Bx=\Bx/\|\Bx\| \in \mathbb{S}^2,
\end{equation}
where $N_1,N_2,N_3\in \mathbb{N}$ and $\alpha_n^m,\beta_n^m,\rho_n^m\in \mathbb C$ are the solutions to the following optimization problem:
\begin{equation}\label{eq:min}
\min_{\alpha_n^m,\beta_n^m,\rho_n^m\in \mathbb C} \Big\| \tilde{\mathbf{H}}-\big(\sum_{n=0}^{N_1}\sum_{m=-n}^{n}\alpha_n^m\Ical_n^m+\sum_{n=1}^{N_2}\sum_{m=-n}^{n}\beta_n^m\Tcal_n^m+\sum_{n=1}^{N_3}\sum_{m=-n}^{n}\rho_n^m\Ncal_n^m\big)
\Big\|_{\hat\Gamma}.
\end{equation}

\begin{rem}
Theoretically, the algorithm developed on partial data is the same as the algorithm developed on full data due to the analyticity of $\tilde{\mathbf{H}}$ on $S_{R_0}$. However, they differ numerically due to the ill-posedness of the inverse problem \eqref{eq:geom3} and the fact that the data extension \eqref{eq:min} shall induce errors to the measurement data.  Nevertheless, as discussed at the end of Section~\ref{subsec:ind}, our reconstruction scheme is robust against noises and hence can overcome this data extension problem. This shall be further corroborated by our numerical experiments in what follows. 
\end{rem}

\section{Numerical experiments and discussions}
\label{sec:experiments}

In this section, we carry out a series of numerical experiments for different benchmark problems to illustrate the salient features of our proposed locating Algorithm~\ref{alg:one}.   The results achieved are consistent with our theoretical predictions in Section~\ref{sec:main} in a sound manner. Besides, the numerical results reveal some very promising potential of the imaging schemes that were not covered in our theoretical analysis and worth further investigation in the future.

The numerical experiments are divided into three groups.
\begin{itemize}
    \item Reconstructions of a single magnetic anomaly.
    \item Reconstructions of multiple magnetic anomalies.
    \item Reconstructions of an extended L-shaped magnetic anomaly.
\end{itemize}

Two types of magnetic anomalies will be considered  in our numerical experiments. They are given by revolving bodies through rotating  in the $x$-$y$ plane around the $x$-axis two 2D boundaries parameterized as follows:
\begin{eqnarray*}
{Peanut:} &\quad & \{ (x,y) : x=\sqrt{3 \cos^2 (s) + 1}\cos(s), \  y=\sqrt{3 \cos^2 (s) + 1}\sin(s), \  0\le s\le 2\pi \},\\
{Kite:} &\quad & \{ (x,y) : x=\cos(s)+0.65\cos (2s)-0.65, \  y=1.5\sin(s), \  0\le s\le 2\pi \}.
\end{eqnarray*}

In our simulations below, the Earth $\Omega$ is simplified to be a unit sphere  with radius $1$ and the Earth's core $\Omega_c$  is a concentric sphere with radius $0.5$. We set $\varepsilon_{c}=4$, $\mu_{c}=10$ and $\sigma_{1}=1$ within the core, $\varepsilon_{s}=1$, $\mu_{0}=1$ within the shell, and $\varepsilon_{l}=2$, $\mu_{l}=4$ and $\sigma_{l}=1$ within the magnetic anomalies. Our measurement data are obtained by solving the Maxwell system \eqref{eq:pss01} using a quadratic
$H(\mathrm{curl})$-conforming edge element discretization in a spherical domain centered at the origin and holding inside the Earth and all the  magnetic anomalies and we deal with the divergence-free constraint using the delta regularization technique (cf.~\cite{Duan12}).  The computational domain is enclosed
by a PML layer to damp the reflection. Local adaptive
refinement scheme   is adopted to
enhance the accuracy of the outgoing scattered wave.  The data
are measured on the the surface of the concentric sphere of radium $R_0=7$, that is $S_{R_0}$ covers the geostationary orbits of the Earth.  The measurement points are taken to be the spherical Lebedev quadrature nodal points on $S_{R_0}$ (cf.\!\cite{Leb99}).
We refine the mesh successively
till the relative maximum error of successive groups of measurement
data is below $0.01\%$.  The exact measurement data
$\bH(\Bx)$ are corrupted point-wise by the formula
\begin{equation}
\bH^{\beta}(\Bx) = \bH(\Bx) + \beta\zeta_1\underset{\Bx}{\max} |\bH(\Bx)|\exp(\mathrm{i} 2\pi
\zeta_2)\,,
\end{equation}
where $\beta$ refers to the relative noise level, and both  $\zeta_1$ and $\zeta_2$ follow the
uniform distribution ranging from $-1$ to $1$. The values of the indicator functions have  been
normalized between $0$ and $1$ to highlight the positions
identified.
The iron liquid within the core are assumed to be constantly rotating along the $z$-axis and thus generating background magnetic field without anomalies in the $y$-$z$ plane as shown in Figure~\ref{fig:background_field}.

\begin{figure}
    \centering
    \includegraphics[width=0.4\textwidth]{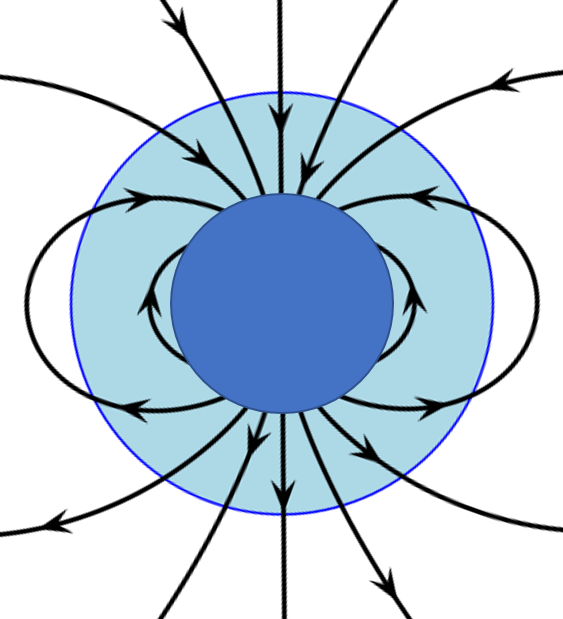}
    \caption{Background magnetic field without magnetic anomalies.}
    \label{fig:background_field}
\end{figure}

\subsection*{Example 1: Single magnetic anomaly}

In this example, we try to identify a single small  magnetic anomaly. We first place a kite magnetic anomaly, scaled by ratio $\delta=0.02$, in the middle between the surfaces of the Earth and its core. From Figure~\ref{fig:singlea}, our locating method works well as predicted and highlights the position of the magnetic anomaly through the normalized indicator function value. The noise level $\beta$ can be increased to 20\% without significantly deteriorating the blow-up behavior around the position in the image, which shows the robustness of our method.

Next, we put a peanut magnetic anomaly, scaled by ratio $\delta=0.01$,  attached to the Earth's surface. The position of the magnetic anomaly still can be correctly retrieved by  the measurement data corrected by $\beta=10\%$ through our method, which even break the limit of the assumption that the anomalies should be deep beneath the surface, see Figure~\ref{fig:singleb}. This shows the potential capability of identifying the magnetic anomaly within the oceans and mountains on the surface of the Earth by employing the new method.

\begin{figure}[!htb]
	\includegraphics[width=0.4\textwidth]{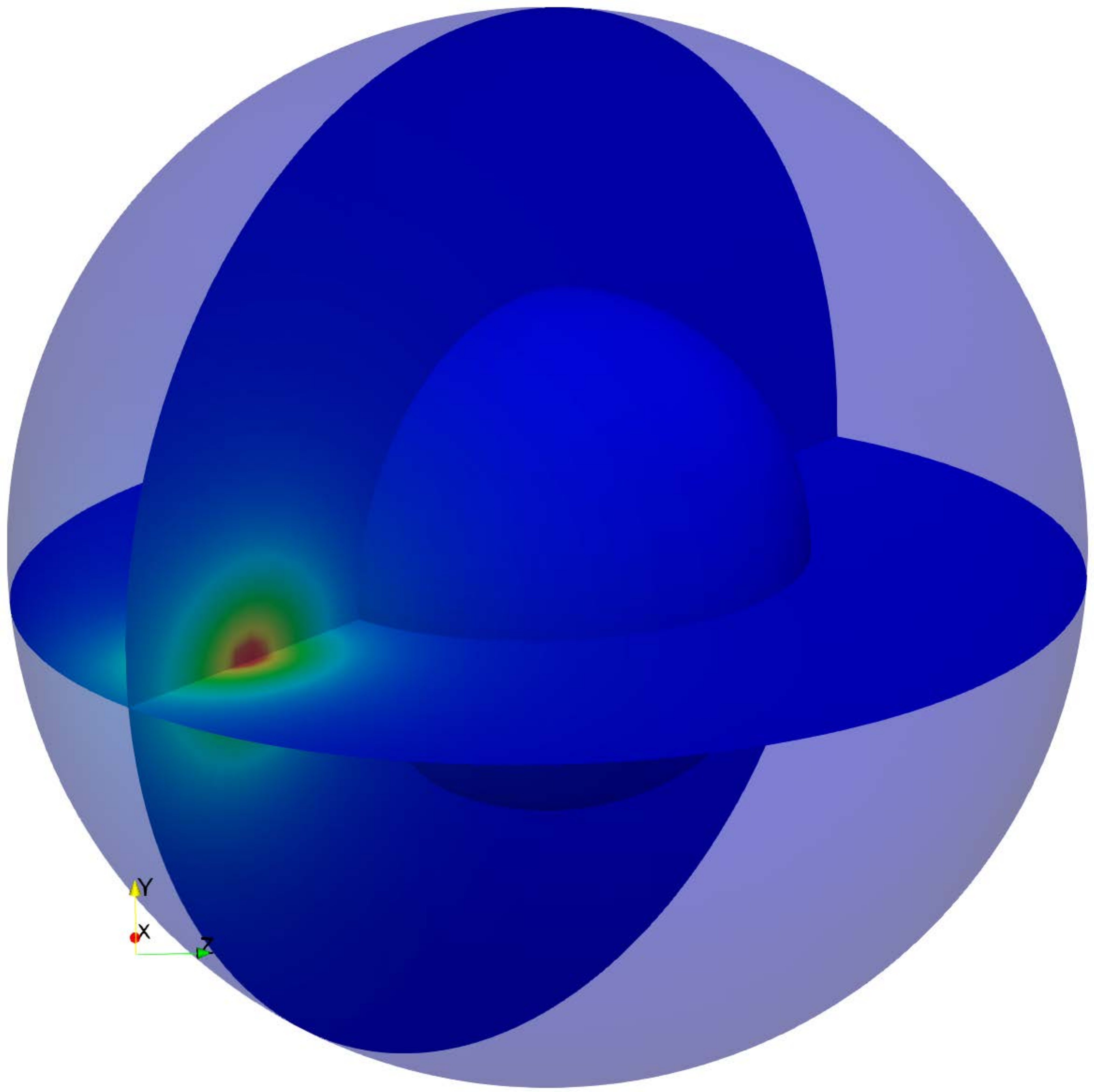}
	\caption{Example 1(a), one small kite magnetic anomaly within the shell of the Earth.} 
	\label{fig:singlea} 
\end{figure}

\begin{figure}[!htb]
	\includegraphics[width=0.4\textwidth]{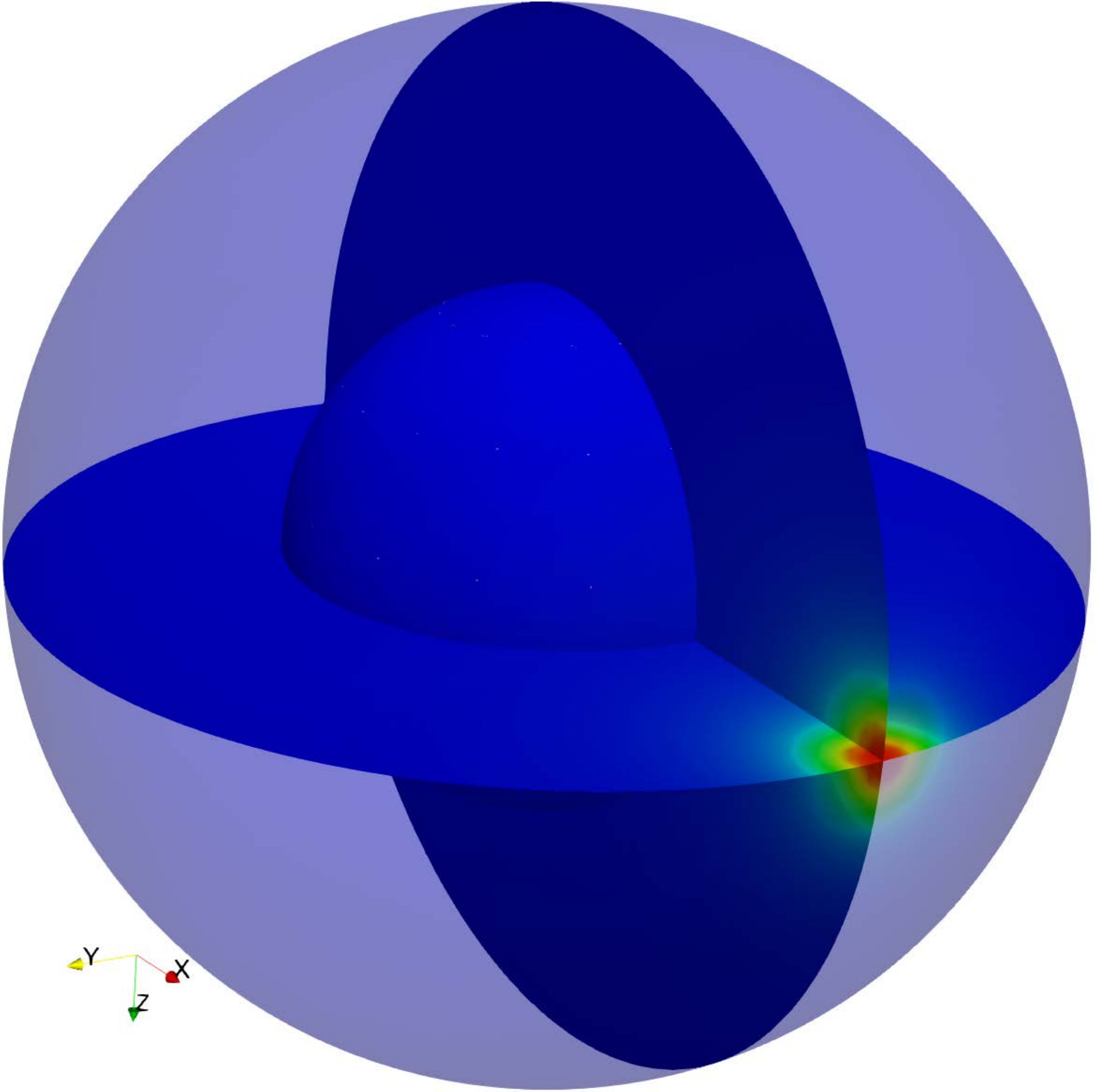}
	\caption{Example 1(b): one small peanut magnetic anomaly beneath the surface of the Earth.}
	\label{fig:singleb} 
\end{figure}

\subsection*{Example 2: Multiple magnetic anomalies}

Four kite-shaped magnetic anomalies, all scaled by $\delta=0.02$, are placed at  well-separated
four positions at different depth within the Earth's shell.
This example verifies that our imaging method can capture multiple magnetic anomalies. The
indicating slice plot of the locating method using measurement
data   are shown in Figure~\ref{fig:multiplea}, which enables one to
identify the locations of all magnetic anomalies even if
the measured data is significantly perturbed with $\beta=10\%$ .

\begin{figure}[!htb]
	\includegraphics[width=.4\textwidth]{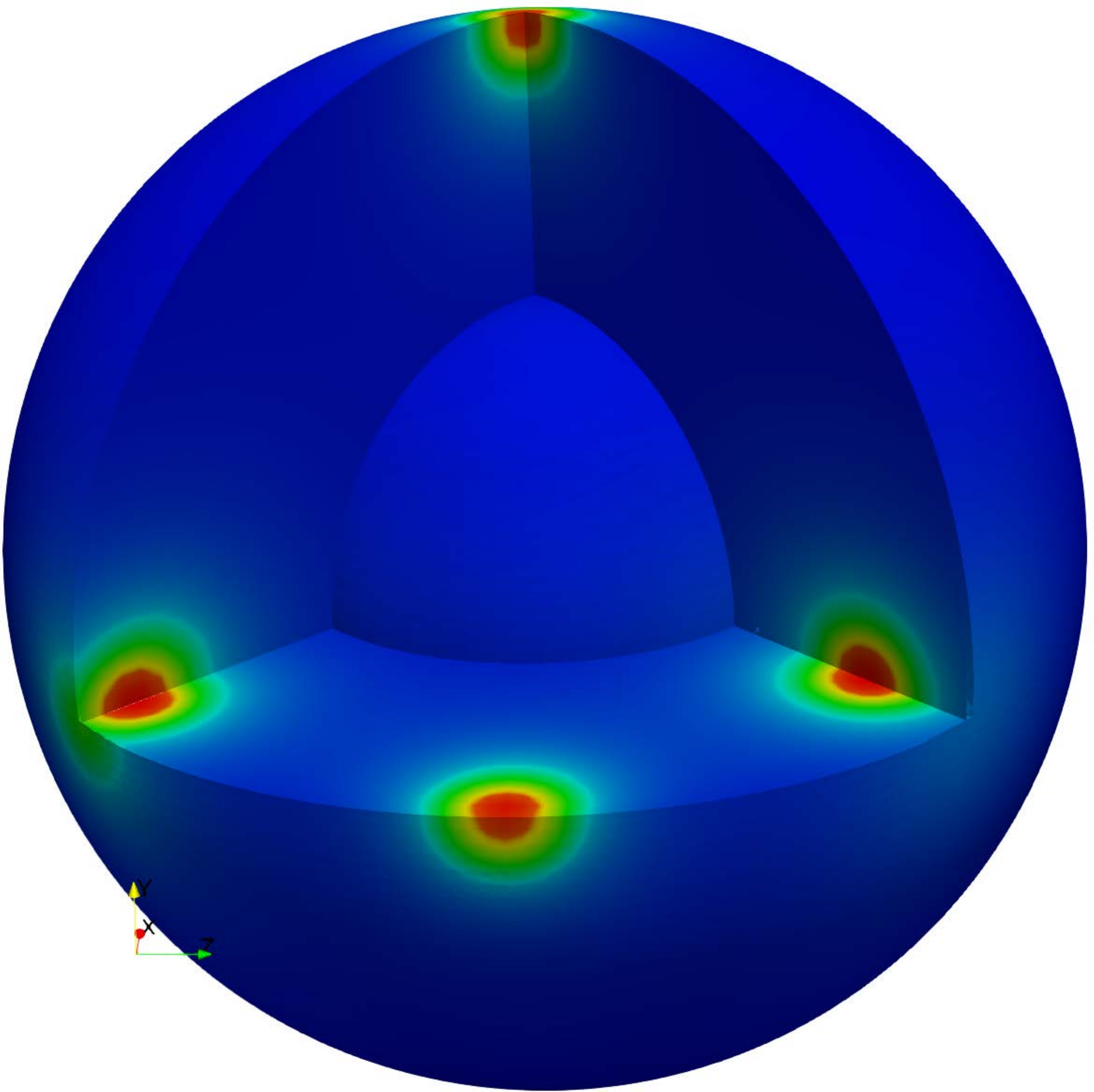}
	\caption{Example 2(a): four separate points buried at different depths within the Earth.}
	\label{fig:multiplea} 
\end{figure}

Next, we illustrate the imaging resolution of  our reconstruction method.
Two peanut-shaped magnetic anomalies, scaled by $\delta=0.03$, are moving towards each other within the same slice plane of the shell. The resolution limit is reached when the highlighted parts of individual magnetic anomalies merge into one
as shown in Figure~\ref{fig:multipleb}. Our reconstruction method can separate these two close magnetic anomalies  well. The location of the magnetic anomalies agrees well with the true ones even under
large noise $\beta=10\%$. When we further reduce the distance between those magnetic anomalies (less than 0.05), the method cannot separate the
two magnetic anomalies any longer, which is below the resolution limit.

\begin{figure}[!htb]
	\includegraphics[width=0.4\textwidth]{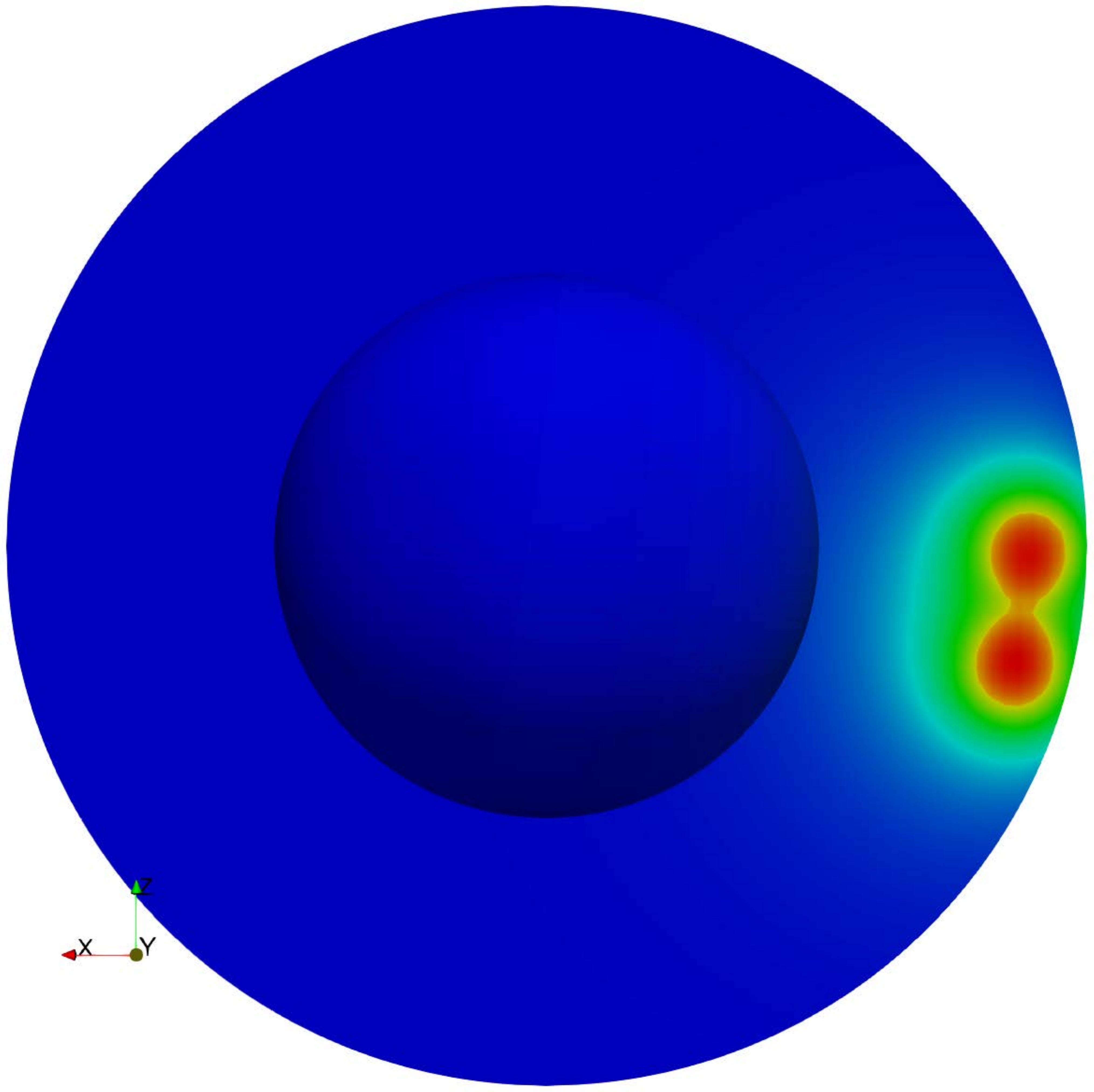}
	\caption{Example 2(b): two points within the Earth to illustrate the resolution limit from top viewpoint.}
	\label{fig:multipleb} 
\end{figure}

\subsection*{Example 3: Extended magnetic anomaly}

Motivated by the merging of blow-up behavior of two close and isolated magnetic anomalies,  We consider in the last example an L-shape
extended magnetic anomaly, which is of size $0.03$ on the cross-section and with armlength $0.5$ along each direction.  This L-shaped magnetic anomaly is more challenging and even out of the scope of the smallness assumption associated with Theorem~\ref{thm:main3}. The reconstruction is shown in Figure~\ref{fig:Lshaped}. One
can see that an L-shaped region is highlighted in red and could be identified with $\beta=10\%$ noise level. This example shows interesting extension of our imaging scheme to determine extended magnetic anomalies with partially small size, like curved line cracks.

\begin{figure}[!htb]
	\includegraphics[width=0.4\textwidth]{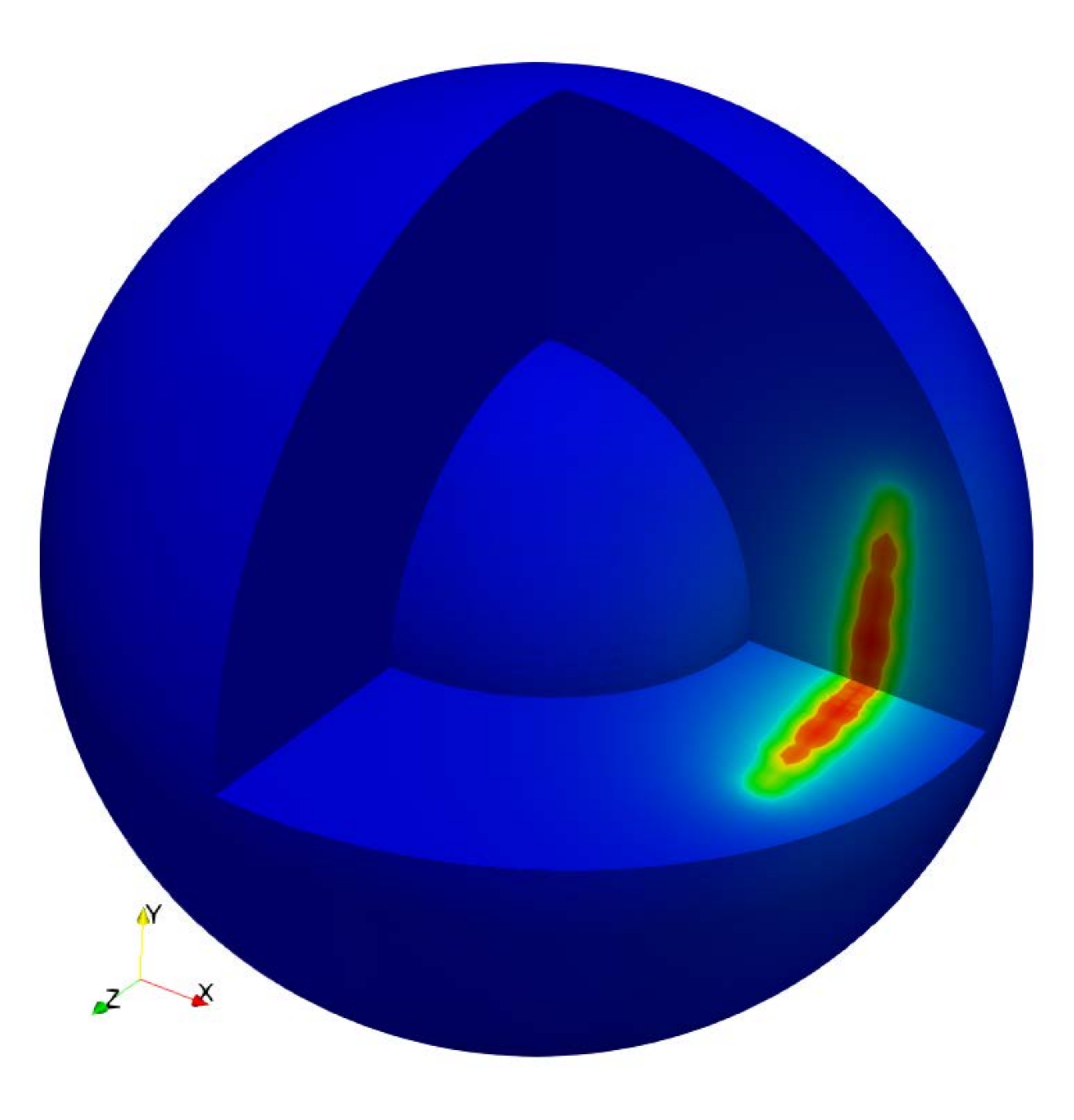}
	\caption{Example 4: an L-shaped line segment within the shell of  the Earth.}
	\label{fig:Lshaped} 
\end{figure}

\subsection*{Example 4: Limited-aperture measurement data}

We consider the reconstruction by using limited-aperture data as discussed in Section~\ref{sect:limitedaperture}. The reconstructed results are shown in Figure \ref{fig:compare} and Table \ref{table1}.

\begin{figure}[!htb]
	\includegraphics[width=0.3\textwidth]{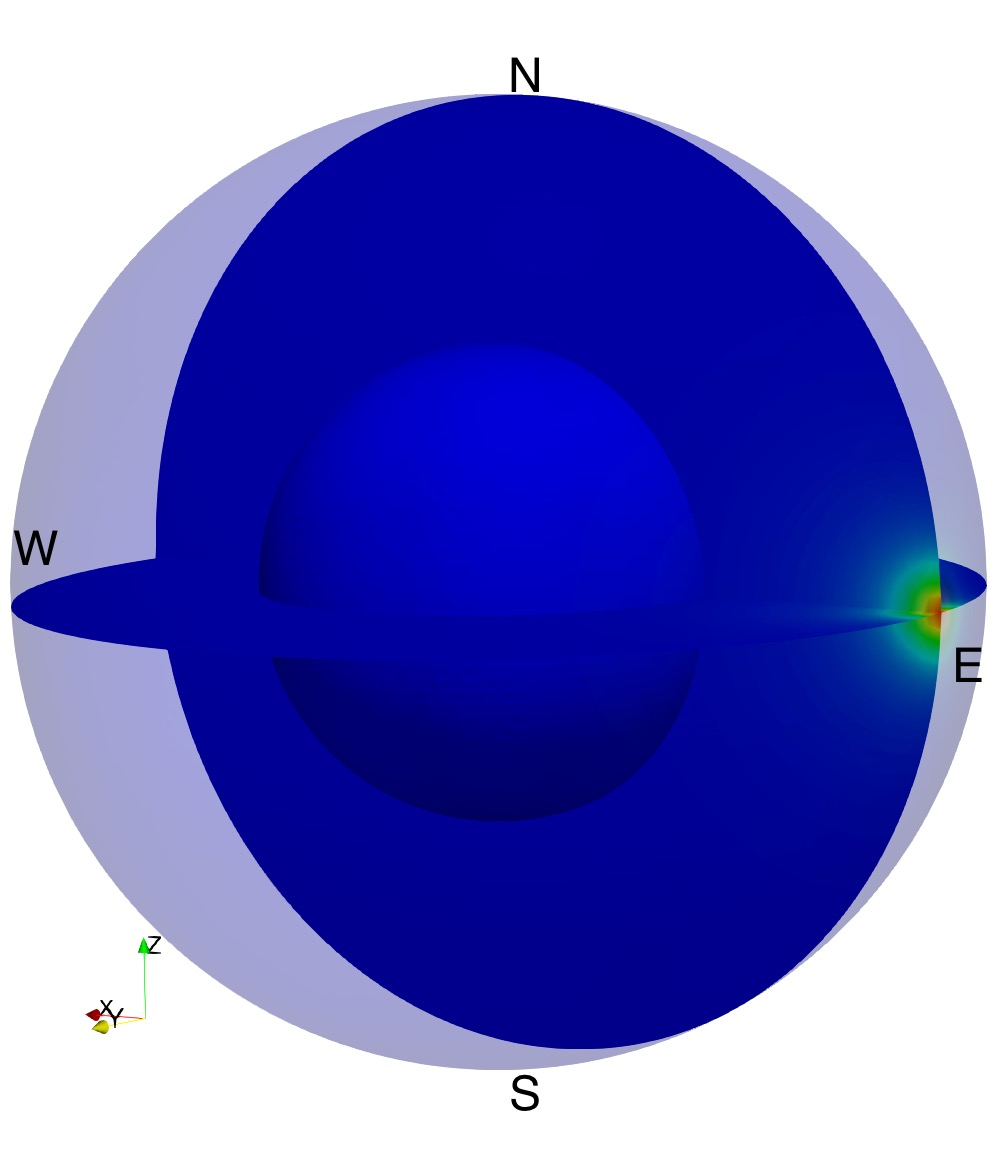}
	\includegraphics[width=0.3\textwidth]{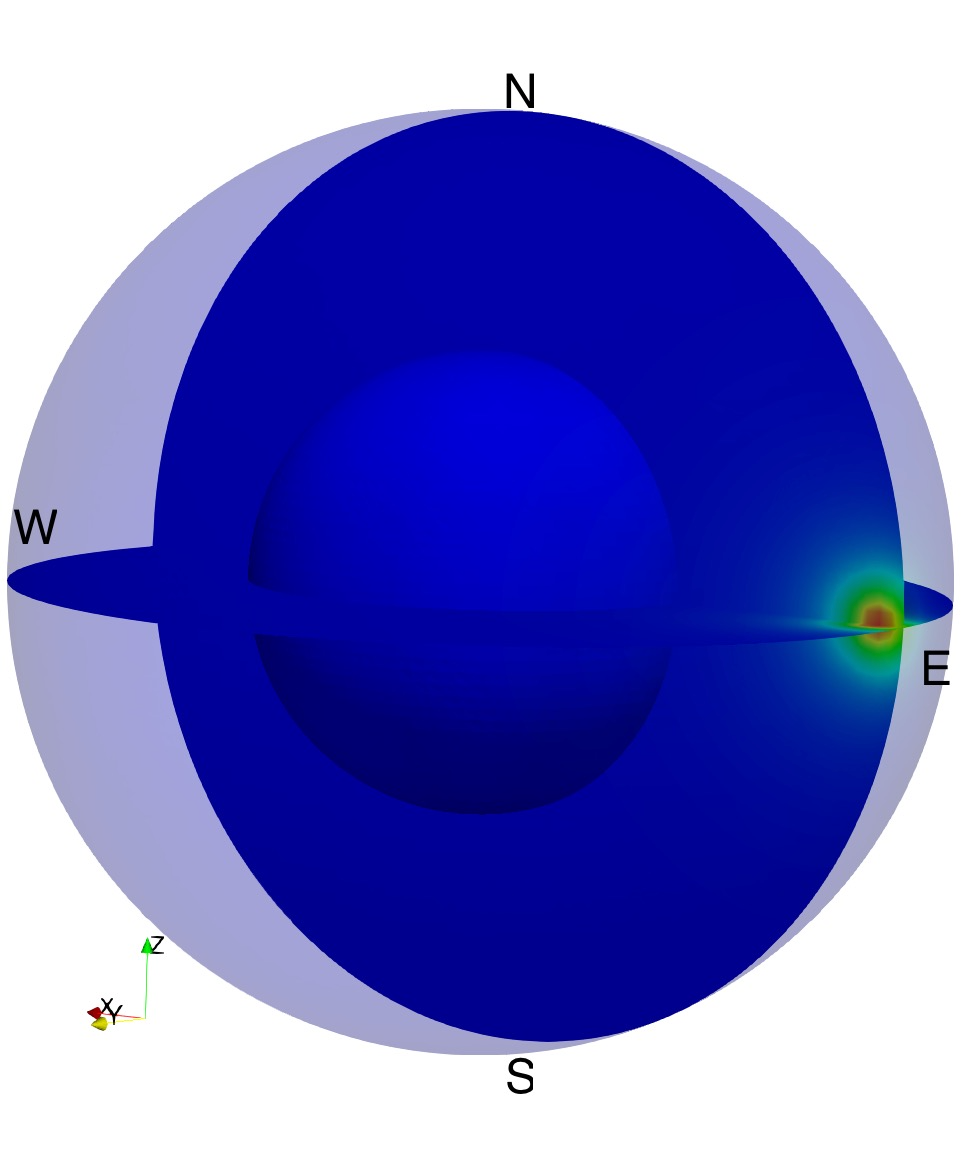}
	\includegraphics[width=0.3\textwidth]{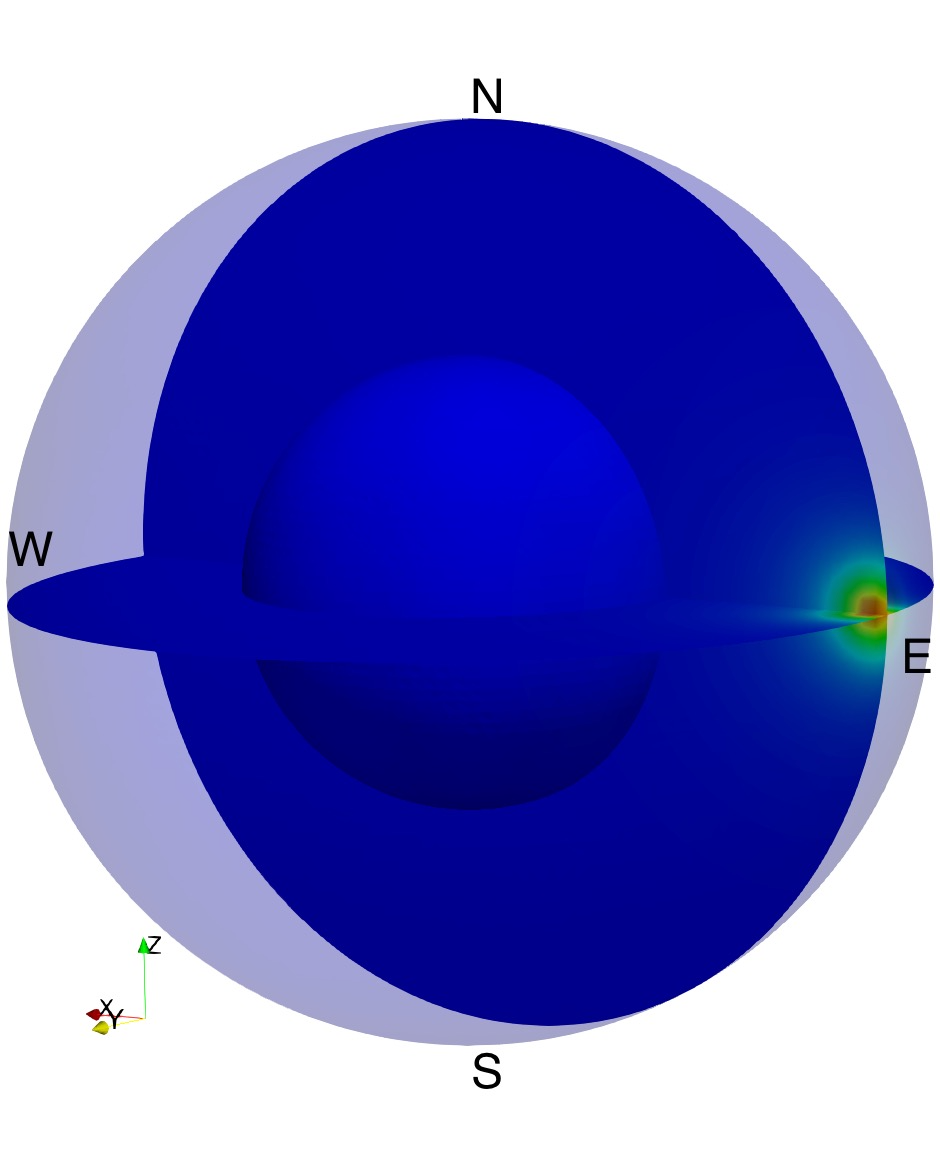}
	\caption{Comparison of the reconstructed results using different data. The data used for the reconstruction are full measurement data on the sphere (left), partial measurement data on the eastern hemisphere (middle), interpolated data from partial measurement data on the eastern hemisphere (right), respectively.}
	\label{fig:compare} 
\end{figure}

\begin{table}[htp]
\caption{The reconstructed positions using different measured data.}
\begin{center}
\begin{tabular}{|>\centering m{0.6\textwidth}|c|}
\hline
\textbf{Data type} & \textbf{Reconstructed position}   \\
\hline
Full measurement data on the sphere with \textbf{1\% noise}. & (0.993, 0.007, -0.008) \\
\hline
Partial measurement data on the eastern \textbf{hemisphere} with \textbf{1\% noise}. & (0.956, -0.022, 0.020) \\
\hline
Interpolated data from partial measurement data on the eastern \textbf{hemisphere} with \textbf{1\% noise}. & (0.974, -0.015, 0.013)\\
\hline
Partial measurement data on the eastern \textbf{quarter-sphere} with \textbf{1\% noise}. & (0.883, -0.062, 0.056) \\
\hline
Interpolated data from partial measurement data on the eastern \textbf{quarter-sphere} with \textbf{1\% noise}. & (0.897, 0.055, -0.052)\\
\hline
Partial measurement data on the eastern \textbf{quarter-sphere} with \textbf{0.1\% noise}. & (0.908, -0.051, 0.045) \\
\hline
Interpolated data from partial measurement data on the eastern \textbf{quarter-sphere} with \textbf{0.1\% noise}. & (0.919, -0.048, 0.041)\\
\hline
Partial measurement data on the eastern \textbf{quarter-sphere} without noise. & (0.922, 0.041, -0.038) \\
\hline
Interpolated data from partial measurement data on the eastern \textbf{quarter-sphere} without noise. & (0.943, -0.028, 0.017)\\
\hline
\end{tabular}
\end{center}
\label{table1}
\end{table}%

\section*{Acknowledgments}

The work of Hongyu Liu is supported by the Hong Kong RGC General Research Funds (projects 12302919, 12301420 and 11300821),  the NSFC/RGC Joint Research Fund (project N\_CityU101/21), the France-Hong Kong ANR/RGC Joint Research Grant, A-HKBU203/19. The work of Li was partially supported by the NSF of China No. 11971221, Guangdong NSF Major Fund No. 2021ZDZX1001, the Shenzhen Sci-Tech Fund No. RCJC20200714114556020, JCYJ20200109115422828 and JCYJ20190809150413261,  National Center for Applied Mathematics Shenzhen, and SUSTech International Center for Mathematics.


\begin{thebibliography}{99}


\bibitem{HK07} {H.~Ammari and H.~Kang}, {\it Polarization and Moment Tensors
With Applications to Inverse Problems and Effective Medium Theory}, Applied Mathematical Sciences, Springer-Verlag, Berlin Heidelberg, 2007.

\bibitem{BPCo96}
G. Backus, R. Parker and C. Constable, {\it Foundations of Geomagnetism}, Cambridge University Press, 1996.

\bibitem{CK} {D.~Colton and R.~Kress}, {\it Inverse Acoustic and Electromagnetic Scattering Theory}, 2nd Edition, Springer-Verlag, Berlin, 1998.

\bibitem{DLL:20}  {Y. Deng, H. Li and H. Liu},{\it Spectral Properties of Neumann-Poincar\'e operator and anomalous localized resonance in elasticity beyond quasi-static limit}, J. Elasticity, {\bf 140} (2020), 213--242.

\bibitem{DLL:19}  {Y. Deng, J. Li and H. Liu}, {\it On identifying magnetized anomalies using geomagnetic monitoring}, Arch. Ration. Mech. Anal., {\bf 231} (2019), 153--187.

\bibitem{DLL20} 
Y. Deng, J. Li and H. Liu, {\it On identifying magnetized anomalies using geomagnetic monitoring within a magnetohydrodynamic model}, Arch. Ration. Mech. Anal., \textbf{235} (2020), no. 1, 691--721.

\bibitem{DHH:17}
{Y. Deng, H. Liu and X. Liu}, {\it Recovery of an embedded obstacle and the surrounding medium for Maxwell's system}, J. Differential Equations, {\bf 267} (4) (2019), 2192--2209.

\bibitem{DLT}
Y. Deng, H. Liu and W. Y. Tsui, {\it Identifying varying magnetic anomalies using geomagnetic monitoring}, Discrete Contin. Dyn. Syst., \textbf{40} (2020), no. 11, 6411--6440. 

\bibitem{DHU:17}
{Y. Deng, H. Liu and G. Uhlmann}, {\it On an inverse boundary problem arising in brain imaging}, J. Differential Equations, {\bf 267} (4) (2019), 2471--2502.

\bibitem{Duan12}  {H. Duan, S. Li,  R. C. E. Tan  and W. Zheng},
 {\it A delta-regularization finite element method for a double curl problem with divergence-free constraint},  SIAM J. Numer. Anal., {\bf 50}  (2012), 3208--3230.
 
 \bibitem{Fey} R. Feynman, R. Leighton and M. Sands, {\it The Feynman Lectures on Physics}, The New Millennium Edition, New York, 2010.

\bibitem{MAD5}
{H. Jin, J. Guo, H. Wang, Z. Zhuang, J. Qin and T. Wang}, {\it Magnetic anomaly detection and localization using orthogonal basis of magnetic tensor contraction}, IEEE Transactions on Geoscience and Remote Sensing, {\bf 58} (2020), 5944--5954.

\bibitem{Jacobs} {J.~A.~Jacobs, Ed.}, {\it Geomagnetism I and II}, Academic Press, Cambridge, 1987.

\bibitem{Leb99} {V.~I.~Lebedev and D.~N.~Laikov}, {\it A quadrature formula for
the sphere of the 131st algebraic order of accuracy}, Doklady
Mathematics,  {\bf 59} (1999), 477--481.

\bibitem{LRX} {H. Liu, L. Rondi and J. Xiao}, {\it Mosco convergence for $H(curl)$ spaces, higher integrability for Maxwell's equations, and stability in direct and inverse EM scattering problems}, J. Eur. Math. Soc., {\bf 21} (2019), no. 10, 2945-2993.

\bibitem{Ned}
{J.~C.~N\'ed\'elec}, {\it Acoustic and Electromagnetic Equations: Integral Representations for Harmonic Problems},
Springer-Verlag, New York, 2001.

\bibitem{MAD1}
{A. Sheinker, B. Ginzburg,  N. Salomonski, P. A. Dickstein, L. Frumkis and B. Z. Kaplan}, {\it Magnetic anomaly detection using high-order crossing method}, IEEE Transactions on Geoscience and Remote Sensing, {\bf 50} (2011), 1095-1103.


\bibitem{MAD3}
{C. Wan, H. Pang, S. Mou, H. Li, M. Pan, Q. Zhang and D. Yang}, {\it Magnetic anomaly detection using a parallel stochastic resonance system}, IEEE Transactions on Instrumentation and Measurement, {\bf 71} (2022), 1--8.

\bibitem{Weiss} N. Weiss, {\it Dynamos in planets, stars and galaxies}, Astronomy and Geophysics, {\bf 43} (2002), 3.09--3.15.

\bibitem{MAD4}
{X. Wu,  S. Huang, M. Li and Y. Deng}, {\it Vector magnetic anomaly detection via an attention mechanism deep-learning model}, Applied Sciences, {\bf 11} (2021), 11533.

\bibitem{MAD2}
{X. Zhang, H. Liu, Z. Wang, H. Dong, J. Ge and Z. Liu}, {\it Anomaly detection of complex magnetic measurements using structured Hankel low-rank modeling and singular value decomposition}, Review of Scientific Instruments, {\bf 93} (2022), 045107.



%
%
%
%
%
%
%
%


%
%

%



%






%
%
%
%





%






%
%
%
%

%
%




%
%


%
%









\end{thebibliography}

\end{document}